\documentclass[oneside,12pt,hidelinks,a4paper]{amsart}

\usepackage{lmodern}
\usepackage[T1]{fontenc}
\usepackage[utf8]{inputenc}
\usepackage[english]{babel} 
\usepackage{etoolbox}
\patchcmd{\subsection}{-.5em}{.5em}{}{}

\usepackage{array}
\usepackage{graphicx}
\usepackage{enumerate}
\usepackage{hyperref}
\usepackage[capitalise]{cleveref}
\usepackage{ifthen}
\usepackage{float}
\newcommand*{\vimage}[1]{\vcenter{\hbox{\includegraphics[height=4cm]{#1}}}}

\usepackage{amsmath}
\usepackage{amssymb}
\usepackage{mathtools}	
\mathtoolsset{centercolon}	
\usepackage{tikz-cd}    
\usetikzlibrary{babel}  
\usepackage{stmaryrd}

\usepackage{tikz}				
\usepackage{tikz-3dplot}
\usetikzlibrary{calc} 
\usetikzlibrary{arrows}	

\usepackage[linesnumbered, lined, boxed, norelsize]{algorithm2e}

\newtheorem{thm}{Theorem}[section]
\newtheorem{prop}[thm]{Proposition}
\newtheorem{lemma}[thm]{Lemma}
\newtheorem{cor}[thm]{Corollary}

\newtheorem*{conjecture*}{Conjecture}

\newtheorem{question}[thm]{Question}
\newtheorem{problem}[thm]{Problem}
\newtheorem*{question*}{Question}
\newtheorem{definition}[thm]{Definition}

\theoremstyle{definition}
\newtheorem{remark}[thm]{Remark}

\newtheorem{ex}[thm]{Example}

\numberwithin{equation}{section}
\numberwithin{figure}{section}
\numberwithin{table}{section}
\numberwithin{subsection}{section}

\newcommand{\N}{\mathbb{N}}
\newcommand{\A}{\mathbb{A}}
\renewcommand{\P}{\mathbb{P}}
\newcommand{\R}{\mathbb{R}}
\newcommand{\C}{\mathbb{C}}
\newcommand{\Q}{\mathbb{Q}}
\newcommand{\Z}{\mathbb{Z}}
\newcommand{\F}{\mathbb{F}}

\newcommand\restr[2]{{
    \left.\kern-\nulldelimiterspace 
    #1
    \right|_{#2} 
}}
\newcommand\supind[1]{{\smash[t]{(#1)}}}

\let\originalleft\left
\let\originalright\right
\renewcommand{\left}{\mathopen{}\mathclose\bgroup\originalleft}
\renewcommand{\right}{\aftergroup\egroup\originalright}

\hypersetup{bookmarksdepth=3}
\usepackage{fancyhdr}
\usepackage{fullpage}
\usepackage{subfigure}
\newenvironment{nouppercase}{
  
  \renewcommand{\uppercasenonmath}[1]{}}{}

\newcommand{\defstyle}{\textit}

\DeclareMathOperator{\Sym}{Sym}
\DeclareMathOperator{\Hom}{Hom}

\DeclareMathOperator{\Aut}{Aut}

\DeclareMathOperator{\id}{id}

\DeclareMathOperator{\diag}{diag}

\DeclareMathOperator{\Fix}{Fix}
\DeclareMathOperator{\im}{im}

\DeclareMathOperator{\rank}{rk}

\DeclareMathOperator{\Pow}{Pow}
\DeclareMathOperator{\Gr}{Gr}
\DeclareMathOperator{\Stab}{Stab}
\DeclareMathOperator{\lcm}{lcm}
\DeclareMathOperator{\Newt}{Newt}

\newcommand{\Orth}[1]{\ensuremath{\mathbb{O}(#1)}}
\newcommand{\SOrth}[1]{\ensuremath{\mathbb{SO}(#1)}}
\newcommand{\SC}[1][n+1]{\ensuremath{\Z_r^{#1}}}
\newcommand{\G}[1][r]{\ensuremath{\mathcal{G}_{#1}}}
\newcommand{\Cx}{\ensuremath{\C[\mathbf{x}]}}
\newcommand{\Cy}{\ensuremath{\C[\mathbf{y}]}}
\newcommand{\Cyt}{\ensuremath{\C[\mathbf{y},t]}}

\DeclareMathOperator{\sym}{\mathfrak{s}}

\newcommand{\powsum}[1]{\ensuremath{\mathfrak{f}}_{#1}}
\DeclareMathOperator{\reci}{\mathcal{R}}
\DeclareMathOperator{\dual}{\mathcal{D}}

\title{\Large Coordinate-wise Powers of Algebraic Varieties}
\author{\large Papri Dey}
\address{\small Indian Statistical Institute (ISI) Kolkata}
\email{\small yedirpap@gmail.com}

\author{\large Paul Görlach}
\address{\small Max-Planck Institute \emph{Mathematics in the Sciences} Leipzig}
\email{\small goerlach@mis.mpg.de}

\author{\large Nidhi Kaihnsa}
\address{\small Department of Mathematical Sciences, University of Copenhagen}
\email{\small nidhi@math.ku.dk}

\begin{document}
\begin{nouppercase}
\maketitle
\end{nouppercase}

\begin{abstract}
  We introduce and study coordinate-wise powers of subvarieties of $\P^n$, 
  i.e.\ varieties arising from raising all points in a given subvariety 
  of $\P^n$ to the $r$-th power, coordinate by coordinate.
  This corresponds to studying the image of a subvariety of $\P^n$ under the 
  quotient of $\P^n$ by the action of the finite group $\Z_r^{n+1}$.
  We determine the degree of coordinate-wise powers and study their defining 
  equations, in particular for hypersurfaces and linear spaces.
  Applying these results, we compute the degree of the variety of 
  orthostochastic matrices and determine iterated dual and reciprocal varieties 
  of power sum hypersurfaces. We also establish a link between coordinate-wise 
  squares of linear spaces and the study of real symmetric matrices with a 
  degenerate eigenspectrum.
  %very degenerate spectrum of eigenvalues.
\end{abstract}

%\tableofcontents

\section{Introduction}

Recently, Hadamard products of algebraic varieties have been attracting 
attention of geometers. These are subvarieties $X \star Y$ of 
projective space $\P^n$ 
that arise from multiplying coordinate-by-coordinate any two points $x \in X$, 
$y \in Y$ in given subvarieties $X,Y$ of $\P^n$.
In applications, they first appeared in \cite{CMS}, where the variety 
associated to the restricted Boltzmann machine was described as a repeated 
Hadamard product of the secant variety of $\P^1 \times \ldots \times \P^1 
\subset \P^{2^n-1}$ with itself.
%In applications, this first appeared in \cite{CMS} where the authors have 
%described the variety associated to the restricted Boltzmann machine as the 
%Zariski closure of the Hadamard product of the secant variety of $\P^1.$
Further study in \cite{BCK}, \cite{FOW}, \cite{BCFL}, \cite{CCFL} made 
progress towards understanding 
%the 
Hadamard products.

Of particular interest are the $r$-th Hadamard powers $X^{\star r} := X \star \ldots \star X$ of an algebraic variety $X \subset \P^n$. They are the multiplicative analogue of secant varieties that play a central role in classical projective geometry: The $r$-th secant variety $\sigma_r(X)$ is the closure of the set of coordinate-wise sums of $r$ points in $X$. Its subvariety corresponding to sums of $r$ equal points is the original variety $X$. In the multiplicative setting, the Hadamard power $X^{\star r}$ replaces $\sigma_r(X)$, but it does not typically contain $X$ if $[1:\ldots:1] \notin X$. As a multiplicative substitute for the inclusion $X \subset \sigma_r(X)$, it is natural to study the subvariety of $X^{\star r}$ given by coordinate-wise products of $r$ equal points in $X$.% $r$-th powers of points in $X \subset \P^n$.

Formally, for a projective variety $X \subset \P^n$ and an integer $r\in\Z$ (possibly negative), we are interested in studying its image under the rational map
  \[\varphi_r \colon \P^n \dashrightarrow \P^n, \qquad [x_0:\ldots:x_n] \mapsto 
  [x_0^r:\ldots:x_n^r].\]
We call the image, $X^{\circ r}$, of $X$ under $\varphi_r$ the $r$-th \defstyle{coordinate-wise power of $X \subset \P^n$}.

\medskip

In this article, we investigate these coordinate-wise powers $X^{\circ r}$ with 
a main focus on the case $r>0$. These varieties show up naturally in many 
applications. For the Grassmannian variety $\Gr(k,\P^n)$ in its Plücker embedding, 
the intersection with its $r$-th coordinate-wise power $\Gr(k,\P^n) \cap 
\Gr(k,\P^n)^{\circ r}$ was described combinatorially in terms of matroids 
in \cite{Len} for even $r$. In \cite{Bon}, highly singular surfaces in $\P^3$ 
have been constructed as preimages of a specific singular surface under the 
morphism $\varphi_r$ for $r>0$. In the case $r=-1$, the map $\varphi_r$ is a classical Cremona transformation and images of varieties under this transformation are called \defstyle{reciprocal varieties} whose study has received particular attention in the 
case of linear spaces, see \cite{LSV}, \cite{KV} and \cite{FSW}. 
%The main focus of our study of the varieties $X^{\circ r}$ lies on the case $r 
%> 0$, but the case of negative coordinate-wise powers is also interesting

For $r > 0$, the coordinate-wise powers $X^{\circ r}$ of a variety $X \subset 
\P^n$ have the following natural interpretation: The quotient of 
$\P^n$ by the finite subgroup $\Z_r^{n+1}$ of the torus $(\C^*)^{n+1}$ is again 
a projective space. The image of a variety $X \subset \P^n$ 
in $\P^n/\Z_r^{n+1} \cong \P^n$ is the variety $X^{\circ r}$, since $\varphi_r 
\colon \P^n \to \P^n$ is the geometric quotient of $\P^n$ by 
$\Z_r^{n+1}$. In other words, coordinate-wise powers of algebraic varieties are 
images of subvarieties of $\P^n$ under the quotient by a certain finite group. 
The case $r=2$ has the special geometric significance of quotienting by the 
group 
generated by reflections at the coordinate hyperplanes of $\P^n$. We are, 
therefore, especially interested in 
\defstyle{coordinate-wise squares} of varieties.

A particular application of interest is the variety of orthostochastic matrices.
%, which is the coordinate-wise square of the variety of 
%orthogonal matrices. In particular, 
An orthostochastic matrix is a matrix arising by squaring each entry of an 
orthogonal matrix. In other words, they are points in the coordinate-wise 
square of the variety of orthogonal matrices. Orthostochastic matrices play a 
central role in the theory of majorization \cite{MOA} and are closely linked to 
finding real symmetric matrices with prescribed eigenvalues and 
diagonal entries, see \cite{Hor} and \cite{Mir}. Recently, it has also been 
shown that studying the variety of orthostochastic matrices is central to the 
existence of determinantal representations of bivariate polynomials and their 
computation, see \cite{Dey}.

As a further application, we show that coordinate-wise squares of linear spaces show up naturally in the study of symmetric matrices with a degenerate spectrum of eigenvalues.

%This project was also motivated for the desire of answer to the question 
%whether a determinantal representation exists for a bivariate polynomial. This 
%problem 
%was effectively reduced to finding the defining equationd of the image of 
%$\mathbb{O}(m)$ under the coordinatewise square. Note that $X^{\circ r}$ has 
%same dimension as $X$. We are mostly interested in exploring the cases where 
%$r>0$, particularly, when $r=2$, i.e.\ the 
%coordinate-wise square of a variety.

\medskip

The article is structured as follows: 
As customary when studying any variety, first and foremost, we 
compute the degree of $X^{\circ r}$. We 
%then explicitly evaluate 
use this to derive
the degree of 
the variety of orthostochastic matrices. In 
Section~\ref{sec:Hypersurfaces}, we dig a little deeper and 
%for hypersurfaces we
find 
%explicit polynomial representation
explicitly the defining equations
of 
the
%its 
%%image under the coordinate-wise square map
coordinate-wise powers
of hypersurfaces.
We define {\em generalised power sum hypersurfaces} and give 
relations between their dual and reciprocal varieties.
%We point out that 
%finding an explicit description of image of an arbitrary variety is not 
%trivial.

We study in more detail 
%images
coordinate-wise powers
of linear spaces in the 
%last
final
section.
We show the dependence of the degree of the coordinate-wise powers of a linear 
space 
on the combinatorial information captured by the corresponding linear matroid.
%corresponding ideal.
%We are particularly interested in squaring 
%We particularly study
Particular attention is drawn to the case of coordinate-wise squares of
linear spaces. 
For 
%low dimensions
low-dimensional linear spaces
we give a complete classification.
% of their images.
We also describe the defining ideal for the coordinate-wise square of general 
linear spaces of arbitrary dimension in a high-dimensional ambient space, and 
we link this question to the 
study of symmetric matrices with a codimension~1 eigenspace.
%with the main theorem
%that describes the ideal of the 
%coordinate-wise square of linear subspace of dimension $k.$
%We conclude 
%with the main theorem
%that describes the ideal of the 
%coordinate-wise square of linear subspace of dimension $k.$

\medskip

\paragraph{\textbf{Acknowledgements}}
The authors would like to thank Mateusz Michałek and Bernd Sturmfels for their 
guidance and suggestions. 
%Authors Paul Görlach and Nidhi Kaihnsa were funded
This work was initiated while the first author was visiting Max Planck 
Institute MiS Leipzig. The financial support by MPI Leipzig which made this 
visit 
possible is gratefully acknowledged. The second and third author were funded
by the International Max Planck Research School {\em Mathematics in the 
Sciences} (IMPRS) during this project.
	
\section{Degree formula} \label{sec:Degree}

%\subsection{Preliminaries and a general degree formula}

Throughout this article, we work over $\C$. 
%We use projective coordinates $x_0,x_1,\ldots,x_n$ on the projective space 
%$\P^n$
We denote the homogeneous coordinate ring of $\P^n$ by $\Cx := 
\C[x_0,\ldots,x_n]$.
% and we denote its homogeneous coordinate ring $ \C[x_0,\ldots,x_n]$ by $\Cx$. 
 For any integer $r \in \Z$, 
we 
consider the rational map
  $$ \varphi_r \colon \P^n \dashrightarrow \P^n, \qquad [x_0:\ldots:x_n] 
  \mapsto [x_0^r:\ldots:x_n^r].$$

For $r \geq 0$, the rational map $\varphi_r$ is a morphism.
%(for $r = 0$, the 
%image of $\varphi_0$ is the point $[1: \ldots : 1]$). 
Throughout, let $X \subset \P^n$ be a projective variety, not necessarily 
irreducible. We denote by $X^{\circ r} \subset \P^n$ the image of $X$ under the 
rational map $\varphi_r$. More explicitly,
  $$X^{\circ r} := \begin{cases}
  \overline{\varphi_r(X \setminus V(x_0 x_1 \ldots x_n))} &\text{if } r < 
  0, \\
  \varphi_r(X) &\text{if } r\geq 0. \\
  \end{cases}$$
For $r < 0$, we will only consider the case that no irreducible component of 
$X$ is contained in any coordinate hyperplane of $\P^n$. We call the image 
$X^{\circ r} \subset \P^n$ the {\em $r$-th coordinate-wise power} 
of $X$. In the case $r=-1$, the variety $X^{\circ (-1)}$ is called the 
{\em reciprocal variety} of $X$. We primarily focus on positive coordinate-wise 
powers in this article, and 
therefore we will from now on always assume $r > 0$ unless explicitly stated 
otherwise. 

Observe 
that $\varphi_r \colon \P^n \to \P^n$ is a finite morphism, and hence, the 
image $X^{\circ r}$ of $X$ under $\varphi_r$ has same dimension as 
$X$. 

The cyclic group $\Z_r$ of order $r$ is identified with the group of
$r$-th roots of unity $\{\xi \in \C \mid \xi^r = 1\}$. We consider the action 
of the $(n+1)$-fold product $\Z_r^{n+1} := \Z_r\times \ldots \times \Z_r$
on $\Cx$ given by rescaling the variables $x_0, \ldots, x_n$ with $r$-th 
roots of unity. We denote the quotient of $\SC$ by the subgroup $\{(\xi,\xi,\ldots,\xi) \in 
\C^r \mid \xi^r = 1\} \subset \SC$ as $\G := \SC/\Z_r$. The group action of 
$\SC$ on $\Cx$ determines a linear action of $\G$ on $\P^n$. In this way, we 
can also view $\G$ as a subgroup of $\Aut(\P^n)$.
%group elements $\tau \in \G$ as certain automorphisms of $\P^n$. 
For $r=2$, this has the geometric interpretation of being the linear group 
action generated by reflections at coordinate hyperplanes. 
Note that $\G$ does not act on the vector space $\Cx_d$ of homogeneous 
polynomials of degree $d$, instead it acts on $\P(\Cx_d)$.

%For homogeneous polynomials $f_1, \ldots, f_m \subset \Cx$, we denote by $(f_1, 
%\ldots, f_m)$ the ideal in $\Cx$ they generate, while we denote 
%by $\langle f_1, \ldots, f_m \rangle$ the linear subspace of the 
%vector space $\Cx$ spanned by $f_1, \ldots, f_m$.
%
%For homogeneous ideals $I, J \subset \Cx$, we say that $I$ and $J$ agree
%\begin{itemize}
%  \item \defstyle{ideal-theoretically} if $I = J$ as ideals.
%  \item \defstyle{scheme-theretically} if $I^{\mathrm{sat}} = 
%  J^{\mathrm{sat}}$, 
%  i.e.\ the saturations of the ideals $I$ and $J$ with respect to the 
%  irrelevant ideal $(x_0,x_1,\ldots,x_n) \subset \Cx$ agree (equivalently, the 
%  projective 
%  schemes defined by $I$ and $J$ are equal: $V(I) = V(J)$ as schemes).
%  \item \defstyle{set-theoretically} if $\sqrt I = \sqrt J$, i.e.\ the radical 
%  ideals of $I$ and $J$ agree (equivalently, $V(I) = V(J)$ as sets).
%\end{itemize}

%\begin{prop}[Dimension]\label{prop:dimension}
%	The dimension of $X^{\circ r}$ equals the dimension of $X$.
%\end{prop}
%
%\begin{proof}
%	Note that $\varphi_r \colon \P^n \to \P^n$ is a finite morphism, hence the 
%	image $X^{\circ r}$ of $X$ under $\varphi_r$ has same dimension as $X$.
%\end{proof}

Given a projective variety, the following proposition describes set-theoretically the preimage 
under $\varphi_r$ of its coordinate-wise $r$-th power.

\begin{prop}[Preimages of coordinate-wise powers] \label{prop:preimage}
	Let $X \subset \P^n$ be a variety and let $X^{\circ r} \subset \P^n$ be its 
	coordinate-wise $r$-th power. The preimage $\varphi_r^{-1}(X^{\circ r})$ is 
	given by $\bigcup_{\tau\in \G}\tau\cdot X$.
\end{prop}

\begin{proof}
	This follows from $X^{\circ r} = \varphi_r(X)$ and the fact that 
	$\varphi_r^{-1}(\varphi_r(p))=\{\tau\cdot p \mid \tau\in \G\}$ for all $p\in 
	X$.
\end{proof}

In particular, for $r=2$, we obtain the following geometric description.

\begin{cor}
	The preimage of $X^{\circ 2}$ under $\varphi_2 \colon \P^n \to \P^n$ is the 
	union over the orbit of $X$ under the subgroup of $\Aut(\P^n)$
%	orbit of $X$ under the group $\G[2]$, whose action on $\P^n$ is
  generated by the reflections in the coordinate hyperplanes.
\end{cor}

In the following theorem, we give a 
%general 
degree formula for the coordinate-wise powers of an irreducible variety.
%image of 
%an irreducible variety under the $r$-th coordinate-wise map.

\begin{thm}[Degree formula] \label{prop:DegreeFormula}
  Let $X \subset \P^n$ be an irreducible projective variety. Let 
  $\Stab_r(X) := \{\tau \in \G \mid \tau \cdot X = X\}$
  and
  $\Fix_r(X) := \{\tau \in \G \mid \restr{\tau}{X} = \id_X\}$.
  Then the degree of the $r$-th coordinate-wise power of $X$ is
    \[\deg X^{\circ r} = \frac{|\Fix_r(X)|}{|\Stab_r(X)|} \: r^{\dim X} \deg 
    X.\]
\end{thm}

\begin{proof}
  Let $H_1, \ldots, H_k \subset \P^n$ for $k := \dim X^{\circ r} = \dim X$ be 
  general hyperplanes whose common intersection with $X^{\circ r}$ consists of 
  finitely many reduced points. We want to determine $|X^{\circ r} \cap 
  \bigcap_{i=1}^k H_i|$. By \cref{prop:preimage}, we have
    \[\varphi_r^{-1}\left(X^{\circ r} \cap \bigcap_{i=1}^k H_i\right) = 
    \bigcup_{\tau \in \G} \tau \cdot \left(X \cap \bigcap_{i=1}^k 
    \varphi_r^{-1} H_i\right).\]
  Note that each $\varphi_r^{-1} H_i$ is a hypersurface of degree~$r$ fixed under the $\G$-action, and their common intersection with $X$ consists of finitely many reduced points by Bertini's theorem (as in \cite[3.4.8]{FOV99}). By Bézout's theorem, $\left|X \cap \bigcap_{i=1}^k \varphi_r^{-1} H_i\right| = r^k \deg X$.
  % many reduced points.
  
  We note that $Z := X \cap \left( \bigcup_{\tau \in \G \setminus \Stab_r(X)} 
  \tau \cdot X \right)$ is of dimension $<k$ by irreducibility of $X$. 
  Therefore, the common 
  intersection of $k$ general hyperplanes $H_i$ with $\varphi_r(Z)$ is empty. 
  This implies that the intersection of $\tau \cdot X$ and $\tau' \cdot X$ does not meet $\bigcap_{i=1}^k \varphi_r^{-1} H_i$ for all $\tau, \tau' \in \G$ with $\tau' \cdot \tau^{-1} \notin \Stab_r(X)$.
%  $\tau \cdot (X \cap \bigcap_{i=1}^k \varphi_r^{-1} H_i) \cap \tau' \cdot (X \cap \bigcap_{i=1}^k \varphi_r^{-1} H_i) = \emptyset$ if and only if $\tau' \cdot \tau^{-1} \notin \Stab_r(X)$.
  %so $Z \cap \bigcap_{i=1}^k \varphi_r^{-1} H_i = \emptyset$. 
%  hence we can write the above as the following disjoint union:
%    \[\bigcup_{\tau \in \G} \tau \cdot \left(X \cap \bigcap_{i=1}^k 
%\varphi_r^{-1} H_i\right) = \bigsqcup_{\tau \in \G/\Stab_r(X)} 
%\tau \cdot \left(X \cap \bigcap_{i=1}^k \varphi_r^{-1} H_i\right).\]
  Hence, the above can be written as a disjoint union
  \[\bigcup_{\tau \in \G} \tau \cdot \left(X \cap \bigcap_{i=1}^k 
  \varphi_r^{-1} H_i\right) = \bigsqcup_{j=1}^s
  \tau_j \cdot \left(X \cap \bigcap_{i=1}^k \varphi_r^{-1} H_i\right),\]
  where $\tau_1, \ldots, \tau_s \in \G$ for $s = |\G|/|\Stab_r(X)|$ represent the cosets of $\Stab_r(X)$ in $\G$.
  
  In particular,
    \[\left|\varphi_r^{-1}\left(X^{\circ r} \cap \bigcap_{i=1}^k 
    H_i\right)\right| = 
    \frac{|\G|}{|\Stab_r(X)|} \: r^k \deg X.\]
  For a general point $p \in X$, we have $\{\tau \in \G \mid \tau \cdot p = p\} = \Fix_r(X)$. Then \cref{prop:preimage} shows that a general point of $X^{\circ r} = \varphi_r(X)$ has $|\G|/|\Fix_r(X)|$ preimages under $\varphi_r$, so for general hyperplanes $H_i$ we conclude
    \[\deg X^{\circ r} = \left|X^{\circ r} \cap \bigcap_{i=1}^k H_i\right| = 
    \frac{|\Fix_r(X)|}{|\G|} 
    \: 
    \left|\varphi_r^{-1}\left(X^{\circ r} \cap \bigcap_{i=1}^k 
    H_i\right)\right| = 
    \frac{|\Fix_r(X)|}{|\Stab_r(X)|} \: r^k \deg X.\]
\end{proof}

\subsection{Orthostochastic matrices} \label{ssec:Orthostochastic}

We use \cref{prop:DegreeFormula} to compute the degree of the variety of 
orthostochastic matrices. By $\Orth{m} \subset \P^{m^2}$ (resp.\ $\SOrth{m} 
\subset \P^{m^2}$) we mean the projective closure of the affine variety of 
orthogonal (resp.\ special orthogonal) matrices in $\A^{m^2}$. It was shown in 
\cite{Dey} that the problem of deciding whether a bivariate polynomial can be 
expressed as the determinant of a definite/monic symmetric linear matrix 
polynomial (a \defstyle{determinantal representation}) is closely linked to the 
problem of finding the defining equations of the variety $\Orth{m}^{\circ 2}$. 
In the case $m=3$, the defining equations of $\Orth{3}^{\circ 2}$ are known 
\cite[Proposition~3.1]{CD} and based on this knowledge, it was shown in 
\cite[Section~4.2]{DeyPhD} how to compute a determinantal representation for a 
cubic 
bivariate polynomial or decide that none exists. For arbitrary $m$, the ideal 
of defining equations may be very complicated, but we are still able to compute 
its degree:
% and is given by the following proposition.

\begin{prop}[Degree of $\Orth{m}^{\circ 2}$] \label{prop:DegreeOrthIm}
  We have $\Orth{m}^{\circ 2} = \SOrth{m}^{\circ 2}$ and its degree is
  \[\deg \Orth{m}^{\circ 2} = 2^{(m-1)^2} \; 
  \frac{\deg\Orth{m}}{2^{\binom{m+1}{2}}} \leq 2^{(m-1)^2}.\]
\end{prop}

\begin{proof}
  The variety $\Orth{m}$ consists of two connected components that are 
  isomorphic to $\SOrth{m}.$ The images of these components under $\varphi_2 
  \colon \P^{m^2} \to \P^{m^2}$ coincide. In particular, $\Orth{m}^{\circ 2} = 
  \SOrth{m}^{\circ 2}$ and $\deg \Orth{m} = 2 \deg \SOrth{m}$. We determine 
  $\Fix_2(\SOrth{m})$ and $\Stab_2(\SOrth{m})$.
  
  Identify elements of $\G[2]$ with $m\times m$-matrices whose entries are $\pm 
  1$. Then a group element $S \in \G[2] = \{\pm 1\}^{m \times m}$ acts 
  on the 
  affine open subset $\A^{m^2} \subset \P^{m^2}$ corresponding to $m \times 
  m$-matrices $M \in \C^{m\times m}$ 
  %by taking the entry-wise product 
  as $S \circ M$, where $S \circ M$ denotes the Hadamard product (i.e.\ 
  entry-wise product) of matrices. Clearly, $\Fix_2(\SOrth{m})$ is trivial, or 
  else every special 
  orthogonal matrix would need to have a zero entry at a certain position.

  We claim that $\Stab_2(\SOrth{m}) \subset \{S \in \{\pm 1\}^{m \times m} \mid 
  \rank S = 1\}$. Indeed, assume that $S \in \{\pm 1\}^{m \times m}$ lies in 
  $\Stab_2(\SOrth{m})$, but is not of rank~$1$. Then $m \geq 2$ and we may 
  assume that the first two columns of $S$ are linearly independent. 
  Consider the vectors $u,v \in \C^m$ given by
  \[u_i := \begin{cases} 1 &\text{if } i < m, \\ -1 &\text{if } i=m \end{cases} 
  \qquad \text{and} \qquad
  v_i := \begin{cases} 2^{i-1} &\text{if } i < m, \\ 2^{m-1}-1 &\text{if } i=m \end{cases} \qquad \text{for all $i \in \{1, \ldots,m\}.$}\]
  Since $u$ and $v$ are orthogonal, we can find a special orthogonal matrix $M 
  \in \C^{m \times m}$ whose first two columns are $M_{\bullet 1} 
  = u/\|u \|_2$ and $M_{\bullet 2} = v/\| v \|_2$. But $S \in 
  \Stab_2(\SOrth{m})$, so the matrix $S \circ M$ must be a special orthogonal 
  matrix. In particular, the first two columns of $S \circ M$ must be 
  orthogonal, i.e.\
  \begin{equation} \label{eq:powersOf2}
  0 = \sum_{i=1}^m (S_{i1}u_i) (S_{i2} v_i) = -(S_{m1} S_{m2}) (2^{m-1}-1) + 
  \sum_{i=1}^{m-1} (S_{i1} S_{i2}) 2^{i-1}.
  \end{equation}
  Since $S_{i1} S_{i2} = \pm 1$ for all $i$, we have $|\sum_{i=1}^{m-1} (S_{i1} 
  S_{i2}) 2^{i-1}| \leq 2^{m-1}-1$, and equality in \eqref{eq:powersOf2} holds if 
  and only if $S_{i1} S_{i2} = S_{j1} S_{j2}$ for all $i,j \in \{1,\ldots,m\}$. 
  However, this contradicts the linear independence of the first two columns of $S$.
  Hence, the claim follows.
  
  Any rank~1 matrix in $\{\pm 1\}^{m\times m}$ can be uniquely written as $u 
  v^T$ with $u,v \in \{\pm 1\}^m$ and $u_1=1$. Such a rank~1 matrix $S = uv^T$ 
  lies in $\Stab_2(\SOrth{m})$ if and only if for each special orthogonal 
  matrix $M \in \C^{m \times m}$ the matrix
    \[S \circ M = (uv^T) \circ M = \diag(u_1,\ldots,u_m) \: M \:  
    \diag(v_1,\ldots,v_m)\]
  is again a special orthogonal matrix. This is true if and only if 
  $\prod_{i=1}^m u_i = \prod_{i=1}^m v_i$. Therefore,
  \begin{align*}
    \Stab_2(\SOrth{m}) &= \{u v^T \mid u,v \in \{\pm 1\}^m, \: u_1=1, \: 
    {\textstyle \prod_i u_i = \prod_i v_i}\},
  \end{align*}
  and, thus, $|\Stab_2(\SOrth{m})| = 2^{2m-2}$.

  Since $\SOrth{m} \subset \P^{m^2}$ is irreducible, applying 
  \cref{prop:DegreeFormula} gives
  \[
  \deg \SOrth{m}^{\circ 2} = \frac{1}{2^{2m-2}} \: 2^{\binom{m}{2}} \deg 
  \SOrth{m} = 2^{\binom{m}{2}-2m+1} \deg \Orth{m}
  %= 2^{\frac{m^2-5m+4}{2}} \deg \SOrth{m} 
  %= 2^{\frac{m^2-5m+2}{2}} \deg\Orth{m} \\
  %= 2^{(m-1)^2-\binom{m+1}{2}} \deg\Orth{m} 
  = 2^{(m-1)^2} \; \frac{\deg\Orth{m}}{2^{\binom{m+1}{2}}}.
  \]
  
  Finally, we observe that the affine variety of orthogonal matrices in 
  $\A^{m^2}$ is an intersection of $\smash{\binom{m+1}{2}}$ 
  quadrics which correspond to the polynomials given by the equation $M^T M = 
  \id$ satisfied by orthogonal matrices $M \in \C^{m \times m}$. Therefore, its 
  projective closure 
  $\Orth{m} 
  \subset \P^{m^2}$ must satisfy $\smash{\deg \Orth{m} \leq 
  2^{\binom{m+1}{2}}}$. This 
  shows $\deg \Orth{m}^{\circ 2} \leq 2^{(m-1)^2}$.
\end{proof}

\begin{remark} \label{rem:DegreeOfOrth}
  The degree of $\Orth{m}$ (resp.\ $\SOrth{m}$) is known for all $m$ by 
  \cite{BBBKR}, namely
  \[\deg \Orth{m} = 2^m \det \left(\: \binom{2m-2i-2j}{m-2i} \:\right)_{1\leq 
    i,j\leq \lfloor m/2 \rfloor}.\]
  Table~\ref{tbl:DegreesOrthIm} shows the resulting degrees of 
  $\Orth{m}^{\circ 2} = \SOrth{m}^{\circ 2}$ for some values of $m$.
  
  \begin{table}[h]
    \boxed
    {
      \begin{tabular}{c|cccccccc}
        $m$ & 1 & 2 & 3 & 4 & 5 & 6 & 7 & 8 \\
        \hline $\deg \SOrth{m}$ & 1 & 2 & 8 & 40 & 384 & 4768 & 111616 & 
        3433600 \\
        $\deg \SOrth{m}^{\circ 2}$ & 1 & 1 & 4 & 40 & 1536 & 152576 & 57147392 
        & 56256102400
      \end{tabular}
    }
    \caption{The degrees of $\SOrth{m}$ and $\SOrth{m}^{\circ 2}$ in 
    comparison.}
    \label{tbl:DegreesOrthIm}
  \end{table}
\end{remark}

\subsection{Linear spaces}%Degrees of powers of linear 
%spaces as matroid invariants}

%We now consider the case that $X = L \subset \P^n$ is a linear space of 
%dimension $k$. 
We now determine the degree of coordinate-wise powers $L^{\circ r}$ for a 
linear space $L \subset \P^n$, based on \cref{prop:DegreeFormula}. It can be 
expressed in terms of the combinatorics captured by the matroid of $L \subset 
\P^n$. We briefly recall some basic definitions for matroids
associated to 
linear 
spaces in $\P^n$. We refer to \cite{Oxl} for a detailed introduction to matroid 
theory.

Let $L \subset \P^n$ be a linear space. The combinatorial information about the 
intersection of $L$ with the linear coordinate spaces in $\P^n$ is captured in 
the \defstyle{linear matroid} $\mathcal M_L$. It is the collection of 
index sets $I \subset \{0,1,\ldots,n\}$ such that $L$ does not intersect 
$V(\{x_i \mid i \notin I\})$. Formally,
\[\mathcal M_L := \{I \subset \{0,1,\ldots,n\} \; \mid \; L \cap V(\{x_i \mid 
i \notin I\}) = \emptyset\}.\]
Different conventions about linear matroids exist in the literature, and some authors take a dual definition for the linear matroid of $L$.

The set $\{0,1,\ldots,n\}$ is the 
\defstyle{ground set} of the matroid. Index sets $I \in \mathcal M_L$ are 
called \defstyle{independent}, while index sets $I \in \Pow(\{0,1,\ldots,n\}) 
\setminus \mathcal M_L$ are called dependent.  An index 
$i \in \{0,1,\ldots,n\}$ is called a \defstyle{coloop} of $\mathcal M_L$ if, 
for all $I \subset 
\{0,1,\ldots,n\}$, the condition $I \in \mathcal M_L$ holds if and only if $I 
\cup \{i\} \in \mathcal M_L$ holds. Geometrically, an index $i \in 
\{0,1,\ldots,n\}$ is a coloop of $\mathcal M_L$ if and only if $L \subset 
V(x_i)$. 

A subset $E \subset \{0,1,\ldots,n\}$ is 
called \defstyle{irreducible} if there is no non-trivial partition $E = E_1 
\sqcup E_2$ with 
\[I \in \mathcal M_L \quad \Leftrightarrow \quad I\cap E_1 \in \mathcal M_L 
\text{ and } 
I\cap E_2 \in \mathcal M_L \qquad \forall I \subset E.\]
The maximal irreducible subsets of $\{0,1,\ldots,n\}$ are called 
\defstyle{components} of $\mathcal M_L$ and they form a partition of 
$\{0,1,\ldots,n\}$. Geometrically, a component of $\mathcal M_L$ is a minimal non-empty
subset of 
$\{0,1,\ldots,n\}$ with the property that $L \cap V(x_i \mid 
i \in I)$ and $L \cap V(x_i \mid i \notin I)$ together span the linear space 
$L$.
%\[L = L \cap V(x_i \mid i \in I) \: \oplus \: L \cap V(x_i \mid i \notin I).\]

In the following result, we determine the degree of $L^{\circ r} \subset \P^n$ 
as an invariant of the linear matroid $\mathcal M_L$.

\begin{thm} \label{prop:LinearDegreeFormula}
  Let $L \subset \P^n$ be a linear space of dimension $k$. Let $s$ be the 
  number of coloops and $t$ the number of components of the associated linear 
  matroid $\mathcal M_L$. Then 
  \[\deg L^{\circ r} = r^{k + s - t + 1}.\]
\end{thm}

\begin{proof}
  By \cref{prop:DegreeFormula}, we need to determine the cardinality 
  of the groups
  \[\Stab_r(L) = \{\tau \in \G \mid \tau \cdot L = L\} \qquad \text{and} \qquad 
  \Fix_r(L) = 
  \{\tau \in \G \mid \restr{\tau}{L} = \id_L\}.\]
  Consider the affine cone over $L$, which is a $(k+1)$-dimensional subspace 
  $W \subset \C^{n+1}$. We denote the canonical basis of $\C^{n+1}$ by $e_0, 
  \ldots, e_n$.
  
  We observe that $|\Fix_r(L)| = |\{\tau \in \SC \mid \restr{\tau}{W} = 
  \id\}|$. For $\tau \in \SC$, we have
  \begin{align*}
  \restr{\tau}{W} = \id \quad 
  &\Leftrightarrow \quad W \subset \langle e_i \mid i \in \{0,1,\ldots,n\} 
  \text{ s.t.\ } \tau_i = 1 \rangle \\
  %    &\Leftrightarrow \quad L \subset V(x_i \mid i \in \{0,1,\ldots,n\} 
  %\text{ 
  %    s.t.\ } \tau_i \neq 1) \\
  &\Leftrightarrow \quad L \subset V(x_i) \quad \forall i \in 
  \{0,1,\ldots,n\} 
  \text{ s.t.\ } \tau_i \neq 1 \\
  %    &\Leftrightarrow \quad i \text{ is a coloop of } \mathcal M_L  \quad 
  %    \forall i 
  %    \in \{0,1,\ldots,n\} \text{ s.t.\ } \tau_i \neq 1 \\
  &\Leftrightarrow \quad \tau_i = 1 \text{ for all $i \in \{0,1,\ldots,n\}$
    which 
    are not a coloop of $\mathcal M_L$}.
  \end{align*}
  From this, we see that $|\Fix_r(L)| = r^s$.
  
  For the stabiliser of $L$, we have $|\Stab_r(L)| = \frac{1}{r}\: |\{\tau \in 
  \SC \mid \tau \cdot W = W\}|$. 
  If $\tau \in \SC$, then
  \begin{align*}
  \tau \cdot W = W \quad
  &\Leftrightarrow \quad W = \bigoplus_{\xi \in \Z_r} W \cap \langle e_i \mid i 
  \in \{0,1,\ldots,n\} \text{ s.t.\ } \tau_i = \xi \rangle \\
  &\Leftrightarrow \quad \text{For each $\xi \in \Z_r$, the set $\{i \in 
    \{0,1,\ldots,n\} \mid \tau_i = \xi\}$ is a union of} \\
  &\phantom{{}\Leftrightarrow{} {}\quad{}} \text{components of $\mathcal M_L$.} 
  \\
  &\Leftrightarrow \quad \forall C \subset \{0,1,\ldots,n\} \text{ component of 
  } 
  \mathcal M_L, \; \exists \xi \in \Z_r \text{ s.t. } \tau_i = \xi \text{ for 
    all } 
  i 
  \in C.
  \end{align*}
  In particular, there are precisely $r^t$ elements $\tau \in \SC$ with $\tau 
  \cdot W = W$. We deduce that $|\Stab_r(L)| = r^{t-1}$, which concludes the 
  proof by \cref{prop:DegreeFormula}.
\end{proof}

\begin{cor} \label{cor:DegreeMatroidInvariant}
  The degree of the coordinate-wise $r$-th power of a linear space only depends 
  on the associated linear matroid. If $L_1, L_2 \subset \P^n$ are linear 
  spaces such that the linear matroids $\mathcal M_{L_1}$ and $\mathcal 
  M_{L_2}$ are isomorphic (i.e.\ they only differ by a permutation of 
  $\{0,1,\ldots,n\}$), then $L_1^{\circ r} \subset \P^n$ and $L_2^{\circ r}
  \subset \P^n$ have the same degree.
\end{cor}

\begin{cor} \label{cor:generalLinearDegree}
  Let $L \subset \P^n$ be a linear space of dimension~$k$. Then $\deg 
  L^{\circ r} \leq r^k$. For general k-dimensional linear spaces in 
  $\P^n$, equality holds.
\end{cor}

\begin{proof}
  Every coloop of $\mathcal M_L$ forms a component of $\mathcal M_L$ and the 
  set $\{0,1,\ldots,n\}\setminus \{\text{coloops}\}$ is a union of components, 
  hence $t\leq s+1$.
  Therefore, by Proposition~\ref{prop:LinearDegreeFormula}, $\deg L^{\circ r} 
  \leq r^k$. 
  A \emph{general} linear space $L \in \Gr(k,\P^n)$ intersects only those linear coordinate space in $\P^n$ of dimension at least $n-k$. Therefore, the linear matroid of a general linear space is the uniform matroid: $\mathcal M_L = \{I \subset \{0,1,\ldots,n\}\, \mid\, |I| \leq n-k-1\}$. It is easily checked from the definitions that this matroid has no coloops and only one component.
\end{proof}

\begin{ex} \label{ex:DegreeHyperplane}
  We illustrate \cref{prop:LinearDegreeFormula} for hyperplanes. Up to 
  permuting and rescaling the coordinates of $\P^n$, each hyperplane is given 
  by 
  $L = V(f)$ with $f = x_0+\ldots+x_m$ for some $m \in \{0,1,\ldots,n\}$. Its 
  linear matroid is 
  \[\mathcal M_L = \{\emptyset, \{0\}, \{1\}, \ldots, \{m\}\}.\]
  The components of this matroid are the set $\{0,1,\ldots,m\}$ and the 
  singletons $\{i\}$ for $i \geq m+1$.
  The matroid $\mathcal M_L$ has no coloops for $m \geq 1$ and the unique 
  coloop $0$ if $m = 0$. Then \cref{prop:LinearDegreeFormula} shows $\deg 
  L^{\circ r} = r^{m-1}$ for $m \geq 1$, and $\deg L^{\circ r} = 1$ for $m = 
  0$. 
  %We observe that this agrees with $\deg f^{\circ r}$ obtained from the 
  %explicit description of $f^{\circ r}$ from \cref{prop:hypersurfaces}. 
  For $m = 3$, $n = 3$ and $r=2$, %we obtain the quartic surface illustrated in 
  we obtain a quartic surface which we illustrate in   
  Figure~\ref{fig:QuarticSurfaceFromPlane}.
\end{ex}

\section{Hypersurfaces} \label{sec:Hypersurfaces}

In this section, we study the coordinate-wise powers of 
hypersurfaces. Here, by a hypersurface, we mean a pure codimension~1 variety. 
In particular, hypersurfaces are assumed to be reduced, but are allowed to have 
multiple irreducible components. We describe a way to find the explicit 
equation describing the image of the given hypersurface under the morphism 
$\varphi_r$. 
We define generalised power sum symmetric polynomials and we give a relation 
between duality and reciprocity of hypersurfaces defined by them. Finally, we 
raise the
%We employ this technique to discuss duality and reciprocity of 
%hypersurfaces defined by power sum symmetric polynomials. Finally, we raise 
%the 
question whether and how the explicit description of coordinate-wise powers of 
hypersurfaces may lead to results on the coordinate-wise powers for arbitrary 
varieties.

\bigskip
\subsection{The defining equation}% of }%\texorpdfstring{$V(f)^{\circ 
%%r}$}{powers 
%		of hypersurfaces}}

The defining equation of a degree~$d$ hypersurface is a square-free (i.e.\ 
reduced) polynomial unique up to scaling, corresponding to a unique 
$f \in \P(\Cx_d)$. We work with points in 
$\P(\Cx_d)$, i.e.\ polynomials up to scaling. We do not always make 
explicit which degree~$d$ we are talking about if it is irrelevant to the 
discussion.
%(only when we do not care what $d$ is -- promised)
The product of $f \in \P(\Cx_d)$ and $g \in \P(\Cx_{d'})$ is 
well-defined up to scaling, i.e.\ as an element $fg \in \P(\Cx_{d+d'})$. 
Equally, we talk about irreducible factors etc.\ of elements of 
$\P(\Cx_d)$.

Since the finite morphism $\varphi_r $ preserves 
dimensions, the coordinate-wise $r$-th power of a hypersurface is again a 
hypersurface, leading to the following definition.

\begin{definition} \label{def:definingEquation}
  Let $f\in \P(\Cx_d)$ be square-free and $V(f) \subset \P^n$ be the 
  corresponding hypersurface. We denote by $f^{\circ r} \in \P(\Cx_{d'})$ 
  the defining equation of the hypersurface $V(f)^{\circ r}$, i.e.\
    \[V(f^{\circ r}) = V(f)^{\circ r}.\]
%  The $r$-th coordinate-wise power $V(f)^{\circ r} 
%  \subset \P^n$ is again a hypersurface. We define $f^{\circ r} \in 
%  \P(\Cx_{d'})$ to be its equation, i.e.\
%    $$V(f^{\circ r}) = V(f)^{\circ r}.$$
\end{definition}

For a 
given square-free polynomial $f$, we want to compute $f^{\circ r}$. To this 
end, we introduce the 
following auxiliary notion.

\begin{definition} \label{def:sym}
	Let $f \in \P(\Cx_d)$ be square-free. We define $\sym_r(f) \in \P(\Cx_{d'})$ 
	as follows:
	\begin{enumerate}[(i)]
		\item If $f$ is irreducible and $f \neq x_i \ \forall i \in 
		\{0,1,\ldots,n\}$, then we define $\sym_r(f) \in \P(\Cx_{d'})$ to be the 
		product over the orbit $\G \cdot f \subset \P(\Cx_d)$. For 
		$f=x_i$, we define $\sym_r(f):=x_i^r.$
		\item If $f=f_1 f_2 \ldots f_m$ where $f_i \in \P(\Cx_d)$ are 
		irreducible, then we define 
      \[\sym_r(f):=\lcm\{\sym_r(f_1), \sym_r(f_2), \ldots, \sym_r(f_m)\}.\]
	\end{enumerate}
\end{definition}

Observe that in case~(ii), determining $\sym_r(f)=\lcm\{\sym_r(f_1), 
\sym_r(f_2), \ldots, \sym_r(f_m)\}$ is straightforward, assuming the 
decomposition of $f$ into irreducible factors $f_1, \ldots, f_m$ is known. 
Indeed, the irreducible factors of each $\sym_r(f_i)$ are immediate from 
case~(i) of the definition, so determining the least common multiple does not 
require any additional factorization.

\begin{lemma} \label{lem:symInvolvesOnlyPowers}
	Let $f\in \P(\Cx_d)$ be square-free. Then $\sym_r(f)\in 
	\P(\C[x_0^r,\ldots,x_n^r]_{d'})$, 
  and the principal ideal generated by $\sym_r(f)$ in the subring 
  $\C[x_0^r,\ldots,x_n^r] \subset \Cx$ is $(f) \cap \C[x_0^r,\ldots,x_n^r]$.
  %$\C[x_0^r,\ldots,x_n^r] \subset \Cx$ is the principal ideal generated by 
  %$\sym_r(f)$:
	%and the intersection of the ideal $(f) \subset \Cx$ with the subring 
	%$\C[x_0^r,\ldots,x_n^r] \subset \Cx$ is the principal ideal generated by 
	%$\sym_r(f)$:
%   \[(f) \cap \C[x_0^r,\ldots,x_n^r] = (\sym_r(f)).\]
\end{lemma}

\begin{proof}
  It is enough to show the claim for $f$ irreducible because we can deduce the 
  general case in the following manner. If $f$ factors into irreducible factors as $f = f_1 
  f_2 \ldots f_m$, then
  \begin{align*}
    (f) \cap \C[x_0^r,\ldots,x_n^r] 
    &= (f_1)\cap \ldots \cap (f_m) \cap \C[x_0^r,\ldots,x_n^r] 
    = \bigcap_{i=1}^m ((f_i)\cap \C[x_0^r,\ldots,x_n^r]) \\
    &= \bigcap_{i=1}^m \, (\sym_r(f_i))
    %&\hspace{-1cm}= (\sym_r(f_1)) \cap \ldots \cap (\sym_r(f_m)) 
    = (\lcm\{\sym_r(f_1), \sym_r(f_2), \ldots, \sym_r(f_m)\}) = (\sym_r(f)).
  \end{align*}
  
  We now assume that $f$ is irreducible. If $f = x_i$ for some $i \in 
  \{0,1,\ldots,n\}$, then the claim holds trivially by the definition of 
  $\sym_r(f)$. Let $f \neq x_i$ for all $i$ and $g$ be a polynomial 
  representing $\sym_r(f) \in \P(\Cx_{md})$. By definition, $\sym_r(f)$ is 
  fixed under the action of $\G$, hence $\tau \cdot g$ is a multiple of $g$ for 
  all $\tau \in \SC$. Since $g$ is not divisible by $x_i$, it must 
  contain a monomial not divisible by $x_i$. This shows that $g$ is fixed by
  $\tau^\supind{i} = (1, \ldots, 1, \zeta, 1, \ldots, 1) \in \SC$, 
  where the $i$-th position of $\tau^\supind{i}$ is a primitive $r$-th root of 
  unity.
  Since $\tau^\supind{0}, \tau^\supind{1}, \ldots, \tau^\supind{n}$ generate 
  the group $\SC$, we have $\tau \cdot g = g$ for all $\tau \in \SC$. 
  Hence, $g$ lies in the invariant ring $\Cx^{\SC} = 
  \C[x_0^r,\ldots,x_n^r]$, i.e.\ $\sym_r(f) \in 
  \P(\C[x_0^r,\ldots,x_n^r]_{d^{\prime}})$.
  
  If $h \in (f)$ is a polynomial in $\C[x_0^r,\ldots,x_n^r]$, then $h$ is 
  invariant under the action of $\SC$ on $\Cx$, so $h \in (\tau \cdot f)$ for 
  all $\tau \in \G$. By the definition of $\sym_r(f)$ and irreducibility of 
  $\tau \cdot f$, this shows $h \in \sym_r(f)$. We conclude $(f) \cap 
  \C[x_0^r,\ldots,x_n^r] = (\sym_r(f))$.
\end{proof}

Based on \cref{def:sym} and \cref{lem:symInvolvesOnlyPowers}, the following 
proposition gives a method to find the 
equation of the coordinate-wise power of a hypersurface.
%image of hypersurface under $\varphi_r$.

\begin{prop}[Powers of hypersurfaces] \label{prop:hypersurfaces}
  Let $V(f) \subset \P^n$ be a hypersurface. The defining equation $f^{\circ 
  r}$ of its 
  coordinate-wise $r$-th power is given by replacing each occurrence of $x_i^r$ 
  in $\sym_r(f)$ by $x_i$ for all $i \in \{0,1,\ldots,n\}$.
\end{prop}

\begin{proof}
  Since $V(f)^{\circ r} \subset \P^n$ is the image of $V(f)$ under $\varphi_r 
  \colon \P^n \to \P^n$, its ideal $(f^{\circ r}) \subset \Cx$ is the preimage 
  under the ring homomorphism $\psi \colon \Cx \to \Cx$, $x_i \mapsto x_i^r$ of 
  the ideal $(f) \subset \Cx$. The claim is therefore an immediate consequence 
  of \cref{lem:symInvolvesOnlyPowers}.
\end{proof}

For clarity, we illustrate the above results for a hyperplane in $\P^3$.
% in Example~\ref{ex:quarticSingularSurface} and for the coordinate-wise square 
%of 
%circles in Example~\ref{ex:SquaringTheCircle}.

\begin{ex} \label{ex:quarticSingularSurface}
  For $n=3$ and $f := x_0+x_1+x_2+x_3 \in \P(\Cx_1)$, we have
    \begin{align*}
      \begin{multlined} \sym_2(f) = (x_0+x_1+x_2+x_3) (x_0+x_1+x_2-x_3) 
      (x_0+x_1-x_2+x_3) (x_0+x_1-x_2-x_3) \\ (x_0-x_1+x_2+x_3) 
      (x_0-x_1+x_2-x_3) (x_0-x_1-x_2+x_3) (x_0-x_1-x_2-x_3). \end{multlined}
%      &\footnotesize 
%      \begin{multlined} = x_0^8 - 4x_0^6x_1^2 + 6x_0^4x_1^4 - 4x_0^2x_1^6 + 
%      x_1^8 - 4x_0^6x_2^2 + 4x_0^4x_1^2x_2^2 + 4x_0^2x_1^4x_2^2 - 4x_1^6x_2^2 
%      + 
%      6x_0^4x_2^4 + 4x_0^2x_1^2x_2^4 + 6x_1^4x_2^4 \\ - 4x_0^2x_2^6 - 
%      4x_1^2x_2^6 + x_2^8 - 4x_0^6x_3^2 + 4x_0^4x_1^2x_3^2 +  4x_0^2x_1^4x_3^2 
%      - 4x_1^6x_3^2 + 4x_0^4x_2^2x_3^2 - 40x_0^2x_1^2x_2^2x_3^2 + 
%      4x_1^4x_2^2x_3^2 + 4x_0^2x_2^4x_3^2 \\ + 4x_1^2x_2^4x_3^2 - 4x_2^6x_3^2 
%      + 
%      6x_0^4x_3^4 + 4x_0^2x_1^2x_3^4 + 6x_1^4x_3^4 + 4x_0^2x_2^2x_3^4 + 
%      4x_1^2x_2^2x_3^4 + 6x_2^4x_3^4 - 4x_0^2x_3^6 - 4x_1^2x_3^6 - 4x_2^2x_3^6 
%      + x_3^8.\end{multlined}
    \end{align*}
  Expanding this expression, we obtain a polynomial in 
  $\C[x_0^2,x_1^2,x_2^2,x_3^2]$ and, substituting $x_i^2$ by $x_i$, we obtain
  by Proposition~\ref{prop:hypersurfaces} that the coordinate-wise 
  square $V(f)^{\circ 2} \subset \P^3$ is the vanishing set of
  \begingroup \scriptsize
    \begin{align*}
%      \begin{multlined}
        f^{\circ 2} &= x_0^4 - 4x_0^3x_1 + 6x_0^2x_1^2 - 4x_0x_1^3 + x_1^4 - 
        4x_0^3x_2 + 4x_0^2x_1x_2 + 4x_0x_1^2x_2 - 4x_1^3x_2 + 6x_0^2x_2^2 + 
        4x_0x_1x_2^2 + 6x_1^2x_2^2 \\
        &- 4x_0x_2^3 - 4x_1x_2^3 + x_2^4 - 4x_0^3x_3 + 4x_0^2x_1x_3 + 
        4x_0x_1^2x_3 - 4x_1^3x_3 + 4x_0^2x_2x_3 - 40x_0x_1x_2x_3 + 4x_1^2x_2x_3 
        + 4x_0x_2^2x_3 \\ 
        &+ 4x_1x_2^2x_3 - 4x_2^3x_3 + 6x_0^2x_3^2 + 4x_0x_1x_3^2 + 6x_1^2x_3^2 
        + 
        4x_0x_2x_3^2 + 4x_1x_2x_3^2 + 6x_2^2x_3^2 - 4x_0x_3^3 - 4x_1x_3^3 - 
        4x_2x_3^3 + x_3^4.
%      \end{multlined}
    \end{align*}
  \endgroup
  \begin{figure}[h]
    \includegraphics[width=0.195\textwidth]{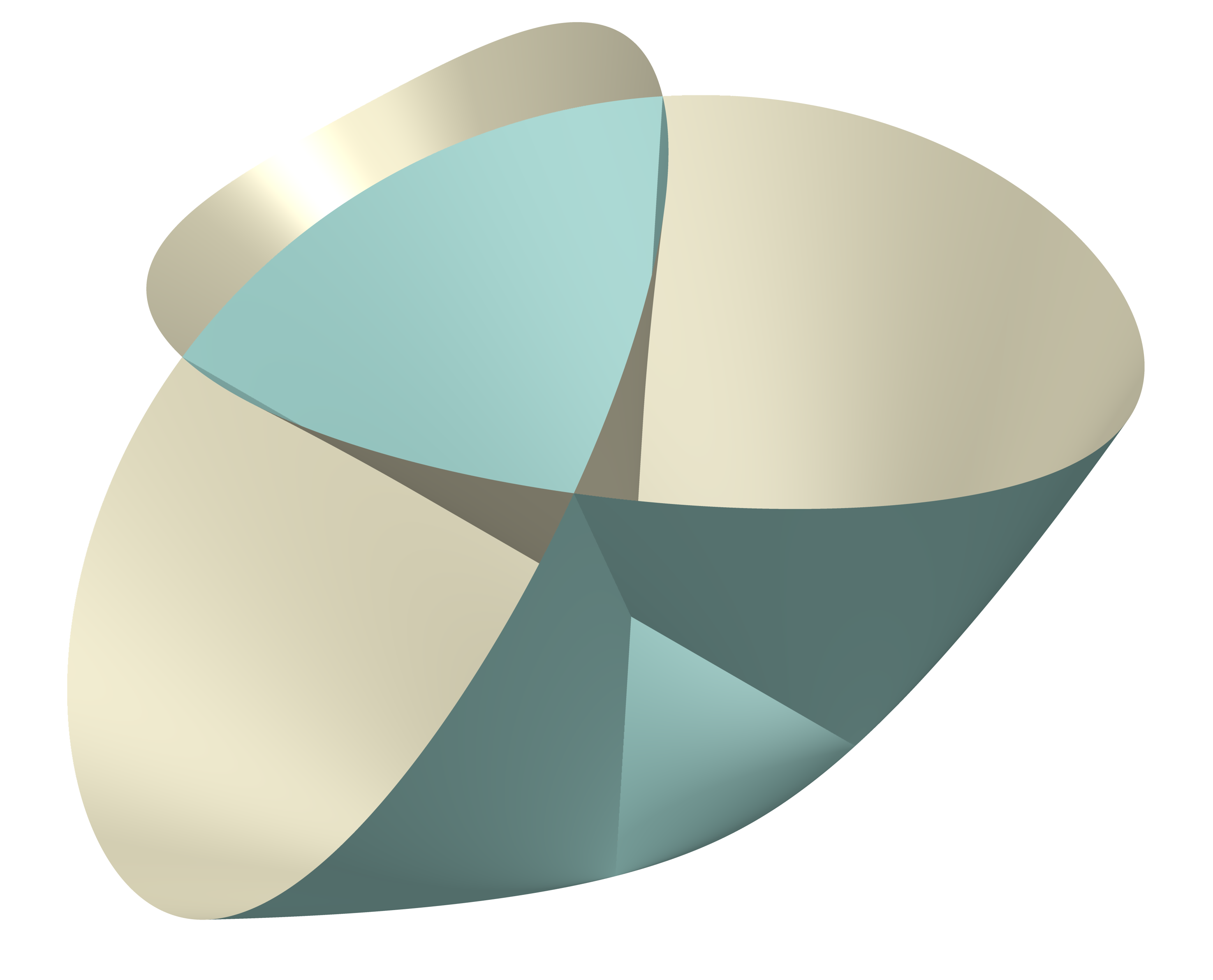}
    \caption{The coordinate-wise square of the plane $V(x_0+x_1+x_2+x_3) 
      \subset \P^3$.}
    \label{fig:QuarticSurfaceFromPlane}
  \end{figure}

  This rational quartic surface is illustrated in 
  Figure~\ref{fig:QuarticSurfaceFromPlane}. It is a Steiner surface with 
  three singular lines forming the ramification locus of 
  $\restr{\varphi_2}{V(f)} \colon V(f) \to V(f)^{\circ 2}$.
\end{ex}

\begin{ex}[Squaring the circle] \label{ex:SquaringTheCircle}
  Consider the plane conic $C = V(f) \subset \P^2$ given by
    $f := (x_1-ax_0)^2+(x_2-bx_0)^2-(c x_0)^2$
  for some $a,b,c \in \R$ with $c > 0$. In the affine chart $x_0=1$, this 
  corresponds over the real numbers to the circle with center $(a,b)$ and 
  radius $c$. From \cref{prop:hypersurfaces}, we show that the 
  coordinate-wise square of the circle $C \subset \P^2$ can be a line, a 
  parabola or a singular quartic curve. See Figure~\ref{fig:SquaringTheCircle} 
  for an illustration of the following three cases:
  \begin{enumerate}[(i)]
    \item If the circle $C$ is centered at the origin (i.e.\ $a=b=0$), then 
    $\sym_r(f) = f$ and
    $C^{\circ 2} \subset \P^2$ is the line defined by the equation
            $f^{\circ 2} = x_1+x_2-c^2x_0.$
    \item If the center of the circle lies on a coordinate-axis and is not the 
    origin (i.e.\ $ab=0$, but $(a,b) \neq (0,0)$), then $C^{\circ 2} \subset 
    \P^2$ is a conic. Say $a=0$, then $C^{\circ 2}$ 
    is defined by the equation
    $f^{\circ 2} = (x_1+x_2)^2 + 2(b^2-c^2)x_0x_1-2(b^2+c^2)x_0x_2 + 
    (b^2-c^2)^2x_0^2.$ In the affine chart $x_0 = 1$, $C$ is a circle and
    $C^{\circ 2}$ is a parabola.
    \item If the center of the circle does not lie on a coordinate-axis, then  
    $|\G \cdot f| = 4$. Therefore, $C^{\circ 2}$ is a quartic plane curve. Its
    equation can be
    computed explicitly using \cref{prop:hypersurfaces}.
%    to be
%    \begingroup \footnotesize
%    \begin{align*}
%    \hspace{0.8cm} f^{\circ 2} &\: = \: 
%    (a^2 + b^2-c^2)^4x_0^4+ x_1^4 + x_2^4+ 4(-a^2 + b^2-c^2)x_0x_1^3 
%    -4(a^2-b^2 + 
%    c^2)(a^2 + b^2-c^2)^2x_0^3x_1 \\ & +
%    2(3a^4-2a^2b^2 + 3b^4 + 2a^2c^2-6b^2c^2 + 3c^4)x_0^2x_1^2 + 
%    4(a^2-b^2-c^2)(a^2 + b^2-c^2)^2x_0^3x_2+ 6x_1^2x_2^2 \\ & + 4(-a^2 
%    + b^2-3c^2)x_0x_1^2x_2 + 4x_1^3x_2 + 2(3a^4-2a^2b^2 + 3b^4-6a^2c^2 + 
%    2b^2c^2 + 3c^4)x_0^2x_2^2  + 4x_1x_2^3 \\&+4(a^2-b^2-3c^2)x_0x_1x_2^2  + 
%    4(a^2-b^2-c^2)x_0x_2^3 -4(a^4 
%    + 10a^2b^2 + b^4 + 2a^2c^2 + 2b^2c^2-3c^4)x_0^2x_1x_2.
%    \end{align*}
%    \endgroup
    Being the image of a 
    conic, the quartic curve $C^{\circ 2}$ is rational, hence it cannot be 
    smooth. In fact, 
    its singularities are the two points $[0:1:-1]$ and
    $[a^2+b^2:b^2(c^2-a^2-b^2):a^2(c^2-a^2-b^2)]$ in $\P^2$. 
    They form the branch 
    locus of $\restr{\varphi_2}{C} \colon C \to C^{\circ 2}$. The point 
    $[0:1:-1] \in \P^2$ is the image of the two complex points $[0:1:\pm i]$ at 
    infinity lying on all of the four conics $\tau \cdot C$ for $\tau \in 
    \G[2]$. The other singular point of $C^{\circ 2}$ is the image under 
    $\varphi_2$ of the two intersection points of the two circles $C$ and $\tau \cdot C$ for $\tau = [1:-1:-1] \in \G[2]$ inside the affine chart $x_0=1.$ 
  \end{enumerate}
  \begin{figure}[h]
    \begin{center}
      ${\vimage{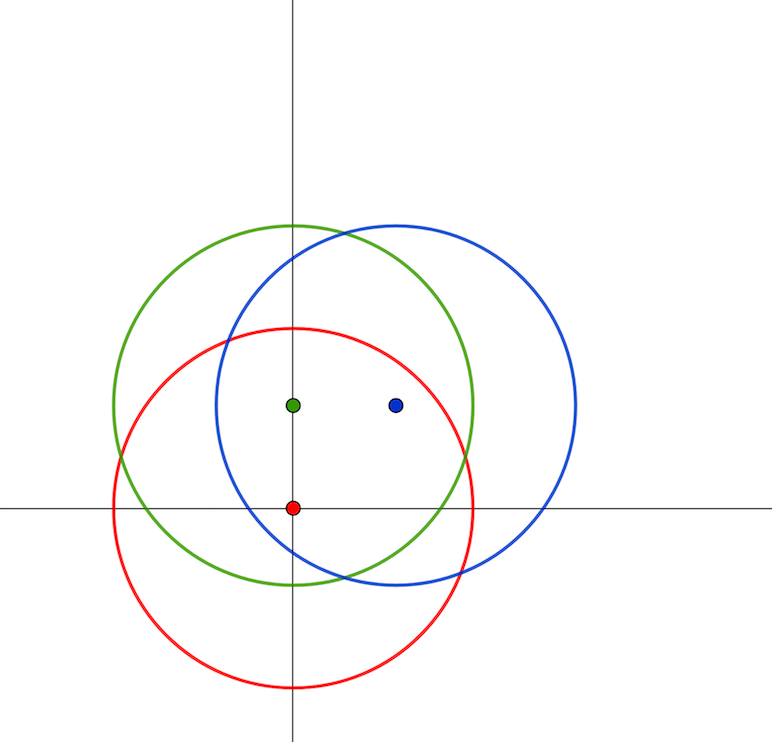}}\qquad
      \longrightarrow 
      \qquad
      {\vimage{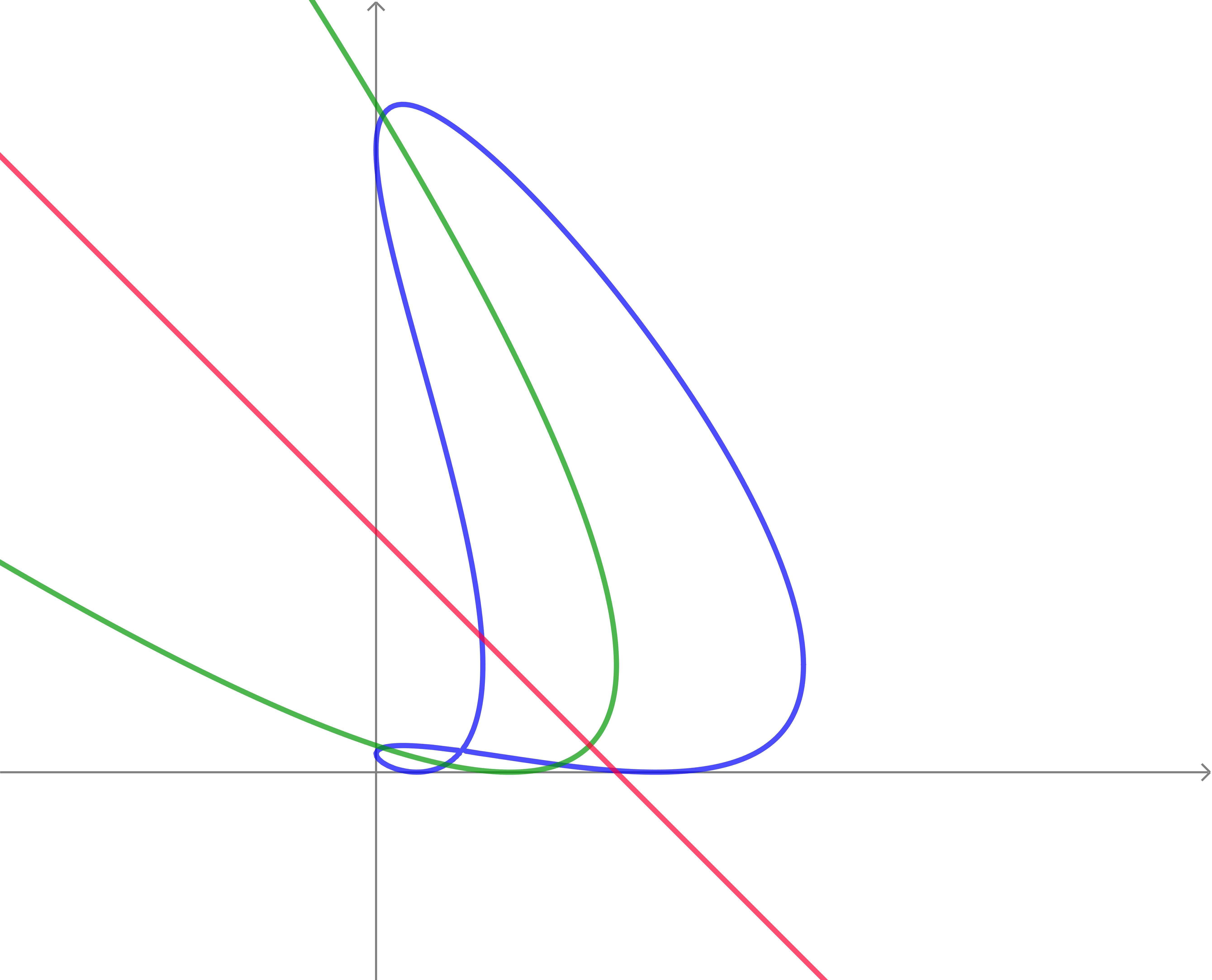}}$
      \caption{Circles and their coordinate-wise squares}
      \label{fig:SquaringTheCircle}
    \end{center}
  \end{figure}
\end{ex}

\begin{remark}[Newton polytope of $f^{\circ r}$]
  Let $f \in \P(\Cx_d)$ be irreducible and $f \neq x_i$. Then the Newton 
  polytope of $f^{\circ 
  r} $ arises from the Newton polytope of $f$ by rescaling 
  according to the cardinality of the orbit $\G \cdot f \subset \P(\Cx_d)$:
    \[\Newt(f^{\circ r}) = \frac{|\G \cdot f|}{r} \cdot \Newt(f) \subset 
    \R^{n+1}.\]
  Indeed, we have $\Newt(\tau \cdot f) = \Newt(f)$ for all $\tau \in \G$, and 
  since $\Newt(g h) = \Newt(g) + \Newt(h)$ holds for all polynomials $g,h$, we 
  have $\Newt(\sym_r(f)) = |\G \cdot f| \cdot \Newt(f)$ by 
  \cref{def:sym}. Replacing $x_i^r$ by $x_i$ rescales the Newton polytope with 
  the factor $\frac{1}{r}$, so the claim follows.
\end{remark}

\bigskip
\subsection{Duals and reciprocals of power sum hypersurfaces}

We now highlight the interactions between coordinate-wise powers, dual 
and reciprocal varieties for the case of \defstyle{power sum hypersurfaces} 
$V(x_0^p+\ldots+x_n^p) \subset \P^n$. Specifically, we determine explicitly 
all hypersurfaces that arise from power sum hypersurfaces by repeatedly taking 
duals and reciprocals as the coordinate-wise $r$-th power of some hypersurface. In this subsection, we also allow $r$ to take negative 
integer values.

Recall that the \defstyle{reciprocal variety} $V(f)^{\circ (-1)}$ of a 
hypersurface $V(f) \subset \P^n$ not containing any coordinate hyperplane of 
$\P^n$ is defined as the closure of 
$\varphi_{-1}(V(f)\setminus V(x_0 x_1 \ldots x_n))$ in $\P^n$. We denote it 
also by $\reci V(f)$. For linear spaces the reciprocal variety and its Chow 
form has been studied in detail in \cite{KV}.
  
We also recall the definition of the \defstyle{dual variety} of $V(f) 
\subset \P^n$. Consider the set of hyperplanes in $\P^n$ that arise as the 
projective tangent space at a smooth point of $V(f)$. This is a subset of the 
dual projective space $(\P^n)^*$ and its Zariski closure is the 
\defstyle{dual variety} of $V(f)$, which we denote by $V(f)^*$ or $\dual 
V(f)$. We identify $(\C^{n+1})^*$ with $\C^{n+1}$ via 
the standard bilinear form and therefore identify $(\P^n)^*$ with $\P^n$.

Consider the \defstyle{power sum polynomial} $\powsum{p} := 
x_0^p+\ldots+x_n^p \in \P(\Cx_p)$ for $p \in \N$. As before, we regard 
polynomials only up to scaling. For power sums with 
negative exponents we consider the numerator of the rational 
function as
  \[\powsum{-p} := (x_1 x_2 x_3 \ldots x_n)^p+(x_0 x_2 x_3 \ldots 
  x_n)^p+\ldots+(x_0 x_1 x_2 \ldots x_{n-1})^p \in \P(\Cx_{np}) \quad 
  \text{for } p \in \N.\]
In particular, $\powsum{-1} \in \P(\Cx_n)$ is the elementary symmetric 
polynomial of degree $n$. %By abuse of notation, we still call the polynomials 
%$\powsum{p}$, for $p<0$, power sum polynomials, even though they are not 
%just sums of powers of variables.

Recall that the morphism $\varphi_r \colon \P^n \to \P^n$ for $r > 0$ is 
finite, hence preserves dimension. Since $\varphi_{-1} \colon \P^n 
\dashrightarrow \P^n$ is a birational map, the rational map 
$\varphi_{-r} = \varphi_{-1} \circ \varphi_r$ also preserves dimensions: $\dim 
V(\powsum{p})^{\circ (-r)} = \dim V(\powsum{p})$. We therefore extend
\cref{def:definingEquation} to include the defining equation of 
$V(\powsum{p})^{\circ r}$ by $\powsum{p}^{\circ r}$ for all $p,r \in \Z 
\setminus \{0\}$. For the constant polynomial $\powsum{0} = 1 \in \P(\Cx_0)
$, we define $\powsum{0}^{\circ r} := 1 $ for all $r 
\in \Z \setminus \{0\}$.
%, still verifying $V(\powsum{0})^{\circ r} = 
%V(\powsum{0}^{\circ r}) \ (= \emptyset)$.

%The following is an elementary observation on coordinate-wise powers of power 
%sum hypersurfaces.

\begin{lemma} \label{lem:powSumRescaling}
  For all $s \in \Z$ and $r, \lambda \in \Z \setminus \{0\}$, we have 
  $\powsum{\lambda s}^{\circ (\lambda r)} = \powsum{s}^{\circ r}$.
\end{lemma}

\begin{proof}
  For $\lambda > 0$, we have $\varphi_{\lambda}^{-1} (V(\powsum{s})) = 
  V(\powsum{\lambda s})$, hence
    \[V(\powsum{\lambda s}^{\circ (\lambda r)}) = \varphi_r(\varphi_\lambda(
    V(\powsum{\lambda s}))) = \varphi_r (V(\powsum{s})) = V(\powsum{s}^{\circ 
    r}),\]
  where we have used the surjectivity of $\varphi_\lambda \colon \P^n \to \P^n$.
  For $\lambda < 0$, we use the above to see
    \[V(\powsum{\lambda s}^{\circ (\lambda r)}) = 
    (V(\powsum{\lambda s})^{\circ (-\lambda)})^{\circ (-r)} = 
    V(\powsum{-s})^{\circ (-r)} =
    (\reci V(\powsum{-s}))^{\circ r}.\]
  The reciprocal variety of $V(\powsum{-s})$ is $V(\powsum{s})$ for all $s \in 
  \Z$. Hence, $V(\powsum{\lambda s}^{\circ (\lambda r)}) = V(\powsum{s})^{\circ 
  r}$.
\end{proof}

This naturally leads us to the our next definition.

\begin{definition}[Generalised power sum polynomial]
  For any rational number $p = \frac{s}{r} \in \Q$ ($r,s \in \Z$, $r \neq 
  0$), we define the \defstyle{generalised power sum polynomial} 
  $\powsum{p} := \powsum{s}^{\circ r} \in \P(\Cx_d)$.
\end{definition}

By Lemma~\ref{lem:powSumRescaling}, the generalised power sum polynomial 
$\powsum{p}$ is 
well-defined. With this definition, we get the following duality result for 
hypersurfaces generalising Example~4.16 in \cite{GKZ}.

\begin{prop}[Duality of generalised power sum hypersurfaces]  
\label{thm:dualPowSum}
  Let $p, q \in \Q \setminus \{0\}$ be such that $\frac{1}{p}+\frac{1}{q} = 1$. 
  Then 
    $V(\powsum{p})^* = V(\powsum{q})$.
\end{prop}

\begin{proof}
  Write $p = \frac{s}{r}$ with $r \in \Z \setminus \{0\}$, $s \in 
  \Z_{>0}$. Let $b \in V(\powsum{p}) = \varphi_r(V(\powsum{s}))$ be a smooth 
  point of $V(\powsum{p}) \setminus V(x_0 x_1 \ldots x_n)$, and let $a \in 
  V(\powsum{s}) \setminus V(x_0 x_1 \ldots x_n)$ be such that $b = 
  \varphi_r(a)$. The morphism $\varphi_r \colon \P^n \setminus V(x_0 x_1 \ldots 
  x_n) \to \P^n \setminus V(x_0 x_1 \ldots x_n)$ induces a linear isomorphism 
  on projective tangent spaces $\mathbb{T}_a \P^n = \P^n \to \P^n = 
  \mathbb{T}_{b} \P^n$ given by $\diag(ra_0^{r-1}, ra_1^{r-1},\ldots, 
  ra_n^{r-1})$. This maps
  \[\mathbb{T}_a V(\powsum{s}) = V\left(\sum_{i=0}^n (\partial_i 
  \powsum{s})(a) \: x_i\right) \subset \P^n \quad \text{onto} \quad
    \mathbb{T}_b V(\powsum{p}) = V\left(\sum_{i=0}^n \frac{(\partial_i 
      \powsum{s})(a)}{ra_i^{r-1}} \: x_i\right) \subset \P^n.\]
  In particular, $V(\powsum{p})^* \subset \P^n$ is the image of the rational map
    \[V(\powsum{s}) \dashrightarrow \P^n, \qquad x \mapsto 
    \left[\frac{\partial_0 
    \powsum{s}}{rx_0^{r-1}}: \frac{\partial_1 
  \powsum{s}}{rx_1^{r-1}}:\ldots:\frac{\partial_n 
  \powsum{s}}{rx_n^{r-1}}\right].\]
  From $\partial_i \powsum{s} = s x_i^{s-1}$ we conclude that 
  $V(\powsum{p})^* = \varphi_{s-r}(V(\powsum{s})) = V(\powsum{s/(s-r)}) = 
  V(\powsum{q})$.
\end{proof}

\begin{remark}
  This statement can be understood as an algebraic analogue of the duality theory for $\ell^p$-spaces $(\R^n, |\cdot|_p)$. Indeed, let $p,q \geq 1$ be rational with $\frac{1}{p}+\frac{1}{q} = 1$. The unit ball in $(\R^n,|\cdot|_p)$ is $U_p := \{v \in \R^n \mid \sum_{i} v_i^p = 1\}$ and, by $\ell_p$-duality, hyperplanes tangent to $U_p$ correspond to the points on the unit ball $U_q$ of the dual normed vector space $(\R^n, |\cdot|_q)$. The complex projective analogues of the unit balls $U_p \subset \R^n$ are the generalised power sum hypersurfaces $V(\powsum{p}) \subset \P^n$ and \cref{thm:dualPowSum} shows the previous statement in this setting.
\end{remark}

Using \cref{prop:hypersurfaces} we can compute $\powsum{p}$ for any $p 
\in \Q$ explicitly. In particular, we make the following observation:

\begin{lemma} \label{lem:genPowSumSubst}
  Let $s \in \N$ and $r \in \Z \setminus \{0\}$ be relatively prime. Then $\powsum{s/r}$ arises 
  from $\powsum{1/r}$ by substituting $x_i \mapsto x_i^s$ for all $i \in 
  \{0,1,\ldots,n\}$.
\end{lemma}

\begin{proof}
  This follows from the explicit description of the polynomials $\powsum{s/r} = 
  \powsum{s}^{\circ r}$ and $\powsum{1/r} = \powsum{1}^{\circ r}$ given by 
  \cref{prop:hypersurfaces}.
\end{proof}

By Lemma~\ref{lem:genPowSumSubst}, in order to determine the generalised power 
sum polynomials $\powsum{p}$, we may restrict our attention to $\powsum{1/r}$. 
These have a particular geometric interpretation as repeated dual-reciprocals 
of the linear space $V(x_0+x_1+\ldots+x_n) \subset \P^n$ as in \cref{repeatedDR}.
%, as the following Corollary of \cref{thm:dualPowSum} shows.

\begin{thm} \label{cor:dualReciprocals}
%  For $r > 0$, the repeated alternating reciprocals and duals of 
%  the linear space $V(\powsum{1}) \subset \P^n$ are the coordinate-wise powers 
%  of $V(\powsum{1})$ given as
%    \begin{align*}
%      {\underbrace{\dual \reci \dual \reci \ldots \dual \reci}_{2r-2}} \:
%      V(\powsum{1}) &= V(\powsum{1})^{\circ r} \qquad 
%      \text{and} \qquad 
%      {\underbrace{\reci \dual \reci \ldots \dual \reci}_{2r-1}} \:
%      V(\powsum{1}) = V(\powsum{1})^{\circ (-r)}.
%    \end{align*}    
%  More generally, 
The repeated dual-reciprocals of generalised power sum 
  hypersurfaces $V(\powsum{p})$ are given by
  \begin{align*}
    (\dual \reci)^k \: V(\powsum{p}) &= V(\powsum{p/(1+kp)}) \qquad \forall k 
    \in \N, \; p \in \Q \setminus \{0, -\frac{1}{k}, -\frac{1}{k-1},\ldots,-1\}
    \quad \text{and} \\
    (\reci \dual)^k \: V(\powsum{p}) &= V(\powsum{p/(1-kp)}) \qquad \forall k 
    \in \N, \; p \in \Q \setminus \{0, \frac{1}{k}, \frac{1}{k-1}, \ldots, 1\}.
  \end{align*}
\end{thm}

%\begin{proof}
%  We show the claim by induction on $r$, with the case $r=1$ amounting to the 
%  observation $\reci V(\powsum{1}) = V(\powsum{-1})$. For $r > 1$, we get by 
%  induction hypothesis,
%  \begin{align*}
%    (\dual \reci)^{r-1} \: V(\powsum{1}) &= \dual \reci V(\powsum{1})^{\circ 
%    (r-1)} 
%    \stackrel{(*)}{=}  V(\powsum{-1/(r-1)})^* 
%    \stackrel{(**)}{=} V(\powsum{1/r}) = V(\powsum{1})^{\circ r},
%  \end{align*}
%  where $(*)$ follows from Lemma~\ref{lem:powSumRescaling} and $(**)$ from 
%  \cref{thm:dualPowSum}. Taking the reciprocal variety of both sides gives the 
%  other identity 
%  $\reci (\dual \reci)^{r-1} \: V(\powsum{1}) = V(\powsum{1})^{\circ (-r)}$.
%\end{proof}
%
%\begin{proof}
%  The claim for $V(\powsum{1})$ follows from the claim for $V(\powsum{p})$ by 
%  setting $k=r-1$ and $p = 1$ resp.\ $p = -1$, and using that $V(\powsum{-1}) 
%  = \reci V(\powsum{1})$.
 \begin{proof} 
  We show the claim for $V(\powsum{p})$ by induction on $k$. For $k=0$, the 
  claim is trivial. For $k > 0$, we get by induction hypothesis:
  \begin{align*}
    (\dual \reci)^k \: V(\powsum{p}) &= \dual \reci V(\powsum{p/(1+(k-1)p)}) 
    = (V(\powsum{p/(1+(k-1)p)})^{\circ (-1)})^* \\ &\stackrel{(*)}{=} 
    V(\powsum{-p/(1+(k-1)p)})^* \stackrel{(**)}{=} V(\powsum{p/(1+kp)}),
  \end{align*}
  where $(*)$ holds by \cref{lem:powSumRescaling} and $(**)$ by 
  \cref{thm:dualPowSum}.
  From this, we also see
    \[(\reci \dual)^k \: V(\powsum{p}) = \reci (\dual \reci)^k \reci 
    V(\powsum{p}) = \reci (\dual \reci)^k V(\powsum{-p}) = \reci 
    V(\powsum{-p/(1-kp)}) = V(\powsum{p/(1-kp)}),\]
  concluding the proof.
\end{proof}

\begin{cor}\label{repeatedDR}
  For $r > 0$, the repeated alternating reciprocals and duals of 
  the linear space $V(\powsum{1}) \subset \P^n$ are the coordinate-wise powers 
  of $V(\powsum{1})$ given as
    \begin{align*}
      {\underbrace{\dual \reci \dual \reci \ldots \dual \reci}_{2r-2}} \:
      V(\powsum{1}) &= V(\powsum{1})^{\circ r} \qquad 
      \text{and} \qquad 
      {\underbrace{\reci \dual \reci \ldots \dual \reci}_{2r-1}} \:
      V(\powsum{1}) = V(\powsum{1})^{\circ (-r)}.
    \end{align*}    
\end{cor}

\begin{ex}
  Let $n = 3$ and $f:=x_0+x_1+x_2+x_3$. The reciprocal variety of the plane $V(f)\subset\P^3$ is 
  given by $\powsum{-1} = x_1 x_2 
  x_3+x_0x_2x_3+x_0x_1x_3+x_0x_1x_2$. 
  Its dual is $V(\powsum{1/2}) = V(\powsum{1})^{\circ 2} \subset \P^3$ by 
  \cref{thm:dualPowSum}. This is the quartic surface from 
  Example~\ref{ex:quarticSingularSurface}. Higher iterated dual-reciprocal 
  varieties of $V(f)$ can be explicitly computed analogous to 
  Example~\ref{ex:quarticSingularSurface} via \cref{cor:dualReciprocals}. For 
  instance, the surface $\dual \reci \dual \reci V(f) \subset 
  \P^3$ is the coordinate-wise cube of $V(f)$ which is the degree~9 
  surface illustrated in Figure~\ref{fig:cubeOfPlane}.
  \begin{figure}[h]
    \includegraphics[width=0.2\textwidth]{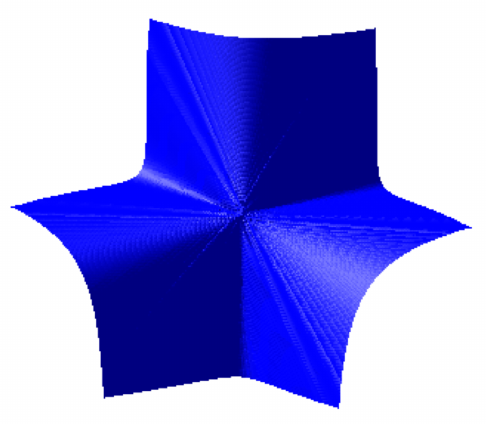}
    \caption{The iterated dual-reciprocal $\dual \reci \dual \reci 
    V(f) \subset \P^3$}% is the coordinate-wise third power of 
    %the plane $V(x_0+x_1+x_2+x_3) \subset \P^3$.}
    \label{fig:cubeOfPlane}
  \end{figure}
\end{ex}

\begin{remark}[Coordinate-wise rational powers]
  The construction of the generalised power sum hypersurfaces $V(\powsum{p})$ may be understood in a broader context of coordinate-wise powers with rational exponents: For a subvariety $X \subset \P^n$, and a rational number $p = r/s$ with $r \in \Z$ and $s \in \Z_{>0}$ relatively prime, we may define the coordinate-wise $p$-th power $X^{\circ p} := \varphi_s^{-1}(X^{\circ r}) = (\varphi_s^{-1}(X))^{\circ r}$. This is a natural generalisation of the coordinate-wise integer powers $X^{\circ r}$. With this definition, %\cref{lem:genPowSumSubst} shows that
  the generalised power sum hypersurface $V(\powsum{p})$ is the $1/p$-th coordinate-wise power of $V(\powsum{1})$. While we focus on coordinate-wise powers to integral exponents in this article, many results easily transfer to the case of rational exponents. For instance, the defining ideal of $X^{\circ (r/s)}$ is obtained by substituting $x_i \mapsto x_i^s$ in each of the generators of the vanishing ideal of $X^{\circ r}$ -- in particular, the number of minimal generators for these two ideals agree.
\end{remark}

\subsection{From hypersurfaces to arbitrary varieties?}

We briefly discuss to what extent \cref{prop:hypersurfaces} can be 
used to determine coordinate-wise powers of arbitrary varieties, and mention 
the difficulties involved in this approach.

Let $r > 0$ and let $f_1, \ldots, f_m$ be homogeneous polynomials vanishing on a variety $X 
\subset \P^n$. Their $r-$th coordinate-wise powers give rise to the inclusion 
$X^{\circ r} \subset V(f_1^{\circ r}, \ldots, f_m^{\circ r})$.
%, since $X \subset V(f_i)$ implies $X^{\circ r} \subset V(f_i)^{\circ r} = 
%V(f_i^{\circ r})$.
We may ask when equality holds, which leads us to the following definition, 
reminiscent of the notion of \emph{tropical bases} in Tropical Geometry 
\cite[Section~2.6]{MS}.

\begin{definition}[Power basis] \label{def:powerBasis}
  A set of homogeneous polynomials $f_1, \ldots, f_m \subset \Cx$ is an 
  $r$-th power basis of the ideal $I = (f_1, \ldots, f_m)$ if the 
  following equality of sets holds:
    \[V(f_1, \ldots, f_m)^{\circ r} = V(f_1^{\circ r}, \ldots, f_m^{\circ r}).\]
\end{definition}

%If $I \subset \Cx$ is the ideal generated by an $r$-th power basis $f_1, 
%\ldots, f_m$, 
%we more specifically say that $f_1, \ldots, f_m$ is an \emph{$r$-th power 
%basis 
%of $I$}.

%\begin{definition}[Power bases] \label{def:powerBasis}
%  Let $I \subset \Cx$ be the ideal of a projective variety $X 
%  \subset \P^n$. A set of homogeneous polynomials 
%  $f_1, \ldots, f_m \in I$ is called an $r$-th power basis of $X$ if
%    \[X^{\circ r} = V(f_1^{\circ r}, \ldots, f_m^{\circ r}).\]
%\end{definition}

%Such a generating set always exists, as we show now.
We show the existence of such power bases for a given ideal in the following 
proposition.

\begin{prop}[Existence of power bases] \label{prop:higherCodimension}
  Let $I \subset \Cx$ be a homogeneous ideal. Then for each $r$, there exists 
  an $r$-th power basis of $I$.
\end{prop}

\begin{proof}
  Let $J$ denote the defining ideal of $V(I)^{\circ r} \subset 
  \P^n$. If $J$ is generated by homogeneous polynomials $g_1, \ldots, g_m \in 
  \Cx$, we define $f_1, \ldots, f_m \in \Cx$ to be their images under the ring 
  homomorphism $\Cx \to \Cx$, $x_i \mapsto x_i^r$. Then 
  $f_i\in I$, since 
    \[V(f_i) =  \varphi_r^{-1}(V(g_i)) \supset 
    \varphi_r^{-1}(V(I)^{\circ r}) \supset 
    V(I).\]
  On the other hand, we have $f_i^{\circ r} = g_i$, since 
  $V(f_i)^{\circ r} = \varphi_r(\varphi_r^{-1}(V(g_i))) = V(g_i)$ by 
  surjectivity of $\varphi_r$. Therefore, $f_1^{\circ r}, \ldots, 
  f_m^{\circ r}$ generate $J$. Enlarging $f_1, \ldots, f_m$ to a 
  generating set of $I$ gives an $r$-th power basis of $I$.
\end{proof}

%\begin{proof}
%  Let $J \subset \Cx$ denote the ideal of $X^{\circ r} \subset \P^n$. 
%  \begin{enumerate}[(i)]
%    \item For all $i \in \{1, \ldots, m\}$, we have $X \subset V(f_i)$, so
%          $X^{\circ r} = \varphi_r(X) \subset \varphi_r(V(f_i)) = V(f_i^{\circ 
%          r})$. Hence, $X^{\circ r} \subset \bigcap_{i=1}^m V(f_i^{\circ r}) = 
%          V(f_1^{\circ r}, 
%          \ldots, f_m^{\circ r})$.
%    \item If $J$ is generated by 
%          homogeneous polynomials 
%          $g_1, \ldots, g_m \in \Cx$, then we define $f_1, \ldots, f_m \in 
%\Cx$ 
%          to be the images of $g_1, \ldots, g_m$ under the ring homomorphism 
%          $\Cx \to \Cx$, $x_i \mapsto x_i^r$. Then $f_i\in I$, since 
%          \[V(f_i) = 
%          \varphi_r^{-1}(V(g_i)) \supset 
%          \varphi_r^{-1}(X^{\circ r}) = \varphi_r^{-1}(\varphi_r (X)) \supset 
%          X.\]
%          On the other hand, we have $f_i^{\circ r} = g_i$, since 
%          $V(f_i)^{\circ r} = \varphi_r(\varphi_r^{-1}(V(g_i))) = V(g_i)$ by 
%          surjectivity of $\varphi_r$. Therefore, $f_1^{\circ r}, \ldots, 
%          f_m^{\circ r}$ generate $J$. \qedhere
%  \end{enumerate}
%\end{proof}

\begin{remark}
  \cref{prop:higherCodimension} shows the existence of $r$-th power bases, but 
  explicitly determining one a priori is nontrivial. For the variety of orthostochastic matrices $\Orth{m}^{\circ 2}$ as in \cref{ssec:Orthostochastic}, it is natural to suspect that the quadratic equations defining the variety of orthogonal matrices would form a power basis for $r = 2$. This is the question discussed in \cite[Section~3]{CD}, where it was shown that this is true for $m = 3$, but not for $m \geq 6$. The cases $m = 4,5$ are an open problem \cite[Problem~6.2]{CD}. Our results on the degree of $\Orth{m}^{\circ 2}$ reduce this open problem to the computation whether explicitly given polynomials $f_1^{\circ 2}, \ldots, f_k^{\circ 2}$ describe an irreducible variety of the correct dimension and degree. Straightforward implementations of this computation seem to be beyond current computer algebra software.
\end{remark}

In the following two 
examples, we will see that even in the case of squaring codimension~2 linear 
spaces, obvious candidates for $f_1, \ldots, f_m$ do not form a power basis.
%only give a strict inclusion $X^{\circ r} 
%\subsetneq V(f_1^{\circ r}, \ldots, f_m^{\circ r})$.

\begin{ex} \label{ex:Codim2GensDontWork}
  Let $I := (f_1, f_2) \subset \Cx$ be the ideal defining the line in $\P^3$ 
  that is given by $f_1 := x_0+x_1+x_2+x_3$ and $f_2 := x_1+2x_2+3x_3$.
%  Consider the line in $\P^3$ given as $X := V(f_1, f_2) \subset \P^3$, 
%  where $f_1 := x_0+x_1+x_2+x_3$ and $f_2 := x_1+2x_2+3x_3$ are generators of 
%  the ideal of $X$. 
  %The orbit of $f_1 \in \P(\Cx_1)$ and $f_2 \in \P(\Cx_1)$ under $\G[2]$ have 
  %sizes~$8$ and $4$, respectively, so by \cref{prop:hypersurfaces}, t
  The polynomials $f_1^{\circ 2}$ and $f_2^{\circ 2}$ have degrees~4 and~2, 
  respectively, by \cref{prop:hypersurfaces}. Note that the polynomial 
  $f_3 := 3x_0^2-x_1^2+x_2^2-3x_3^2 = 3(x_0-x_1-x_2-x_3)f_1+2(x_1+x_2)f_2$ also 
  lies in $I$, so the ideal of $V(I)^{\circ 2}$ contains the linear form 
  $f_3^{\circ 2} = 3x_0-x_1+x_2-3x_3$. The polynomials $f_1,f_2$ do not form a 
  power 
  basis of $I$. In fact, one can check that $V(f_1^{\circ 2}, f_2^{\circ 2}) 
  \subset \P^3$ is the union of four rational quadratic curves, one of which 
  is $V(I)^{\circ 2}$, see Figure~\ref{fig:4Conics} for an 
  illustration. A power basis of $I$ is given by $f_1,f_2,f_3$.
   \begin{figure}[h]
    \includegraphics[width=0.2\textwidth]{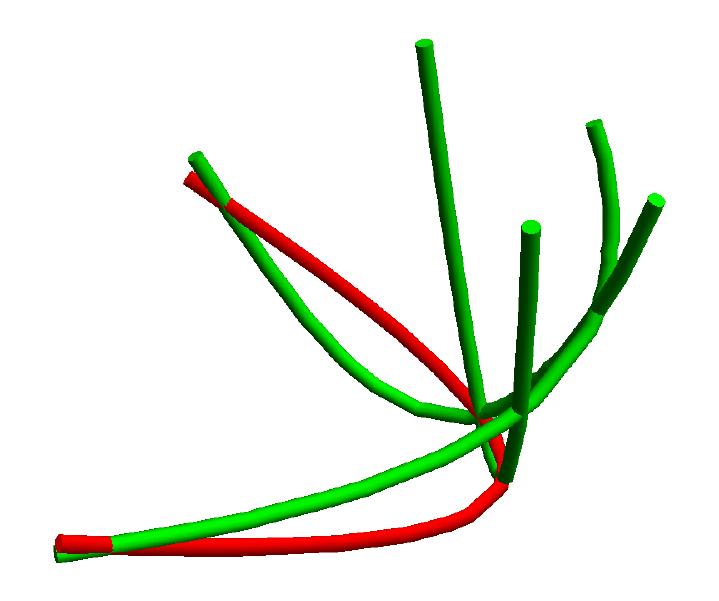}
    \caption{Distinction between $V(f_1^{\circ 2}, f_2^{\circ 2})$ and 
    $V(f_1,f_2)^{\circ 2}$}
%    \caption{$V(f_1^{\circ 2}, f_2^{\circ 2}) \subset \P^3$ is a union of 
%    four embedded plane conics, while $V(f_1,f_2)^{\circ 2} \subset \P^3$ is 
%    only one of them.}
    \label{fig:4Conics}
  \end{figure}
\end{ex}

\begin{ex} \label{ex:Codim2CircuitsDontWork}
  Another natural choice for polynomials $f_1, \ldots, f_m$ in the ideal of a 
  linear space $X \subset \P^n$ consists of the \emph{circuit forms}, i.e.\ 
  linear forms vanishing on $X$ that are minimal with respect to the set of 
  occurring variables. However, for 
  \[X := V(x_0+x_1+x_2+x_3+x_4, \; x_1+2x_2+3x_3+4x_4) \subset \P^4,\]
  these circuit forms are
  \begin{gather*}
    f_1 = x_1+2x_2+3x_3+4x_4, \quad f_2 = x_0-x_2-2x_3-3x_4, \quad f_3 = 
  2x_0+x_1-x_3-2x_4, \\
   f_4 = 3x_0+2x_1+x_2-x_4, \quad f_5 = 
  4x_0+3x_1+2x_2+x_3,
  \end{gather*}
  and one can check that the point $[16:16:1:36:9] \in \P^4$ lies in 
  $V(f_1^{\circ 2}, \ldots, f_5^{\circ 2})$, but not in $X^{\circ 2}$. In 
  particular, $f_1, \ldots, f_5$ is not an $r$-th power basis for $r = 2$.
\end{ex}

We have seen in \cref{ex:Codim2GensDontWork} and 
\cref{ex:Codim2CircuitsDontWork} that even for the case of linear spaces of 
codimension~2 it is not an easy task to a priori identify an $r$-th power 
basis.

%In general, it would be desirable to find an $r$-th power basis for 
%specific ideals $I$ without first computing the ideal of $V(I)^{\circ r}$.
The following proposition shows how one can straightforwardly find a very large
$r$-th power basis of an ideal $I$, without first computing the ideal of 
$V(I)^{\circ r}$.

\begin{prop} \label{prop:stupidPowerBasis}
  If $g_1, \ldots, g_k \in \Cx_d$ are forms of degree~$d$, then taking 
  $(k-1)r^n+1$ general linear combinations of $g_1, \ldots, g_k$
  produces an $r$-th power basis of $(g_1, \ldots, g_k)$.
\end{prop}

\begin{proof}
  We assume that $g_1, \ldots, g_k$ are linearly independent, or else we 
  can replace them with a linearly independent subset. For $m := (k-1)r^n+1$, 
  let 
  $f_1, \ldots, f_m \in 
  \langle g_1, \ldots, g_k \rangle$ be such that no $k$ of them are linearly 
  dependent.
  %Let $I := (g_1, \ldots, g_k) \subset \Cx$ and $X := V(I)$.
  For $X := V(g_1, \ldots, g_k)$, we
  will show that $V(f_1^{\circ r}, \ldots, f_m^{\circ r}) = X^{\circ r}$ by 
  comparing the preimages of both sides under $\varphi_r \colon \P^n \to \P^n$. 
  
  By Proposition~\ref{prop:preimage}, we have $\varphi_r^{-1}(X^{\circ r}) = 
  \bigcup_{\tau \in \G} \tau \cdot X$ and
  \[\varphi_r^{-1}(V(f_1^{\circ r}, \ldots, f_m^{\circ r})) = 
  \bigcap_{i=1}^m \varphi_r^{-1} (\varphi_r(V(f_i))) = 
  \bigcap_{i=1}^m \bigcup_{\tau \in \G} \tau \cdot V(f_i).\]
  Let $p \in \varphi_r^{-1} (V(f_1^{\circ r}, \ldots, f_m^{\circ r})) \subset 
  \P^n$. Then for each $i \in  \{1,\ldots,m\}$ there exists some $\tau \in \G$ 
  with $p \in \tau \cdot V(f_i)$ using the last equality above. Since $m > (k-1)|\G|$, by pigeonhole principle there must exist $\tau 
  \in \G$ and $i_1, i_2,\ldots, i_k \in \{1,\ldots,m\}$ distinct with $p \in 
  \bigcap_{j=1}^k \tau \cdot V(f_{i_j}) = \tau \cdot 
  V(f_{i_1},\ldots,f_{i_k})$. Since,  by assumption, no $k$ of them are linearly dependent $f_{i_1}, \ldots, f_{i_k}$ span $\langle g_1, 
  \ldots, g_k\rangle$. Therefore, $V(f_{i_1},\ldots,f_{i_k}) = X$, and hence, $p\in\tau \cdot 
  V(f_{i_1},\ldots,f_{i_k})$
  implies that $p \in \tau \cdot X \subset \varphi_r^{-1}(X^{\circ r})$. 
  This shows $\varphi_r^{-1} (V(f_1^{\circ r}, \ldots, f_m^{\circ r})) \subset 
  \varphi_r^{-1}(X)$. The reverse inclusion is trivial.
\end{proof}

In particular, \cref{prop:stupidPowerBasis} shows that for a subvariety of 
$\P^n$ defined by $k$ forms of degree~$d$, its coordinate-wise $r$-th power can 
be described set-theoretically by the vanishing of $(k-1)r^n+1$ forms of degree 
$\leq d r^{n-1}$. However, we will see in \cref{sec:LinSpaces} that for linear 
spaces this bound is rather weak in many cases and should be expected to allow 
dramatic refinement in general. We raise the following as a broad open question:

\begin{question}
  When does a set of homogeneous polynomials form an $r$-th 
  power basis? For a given ideal $I$, do there exist polynomials $f_1, \ldots, f_m \in I$ that simultaneously form an $r$-th power basis for all $r$?
\end{question}

\section{Linear spaces} \label{sec:LinSpaces}

In this section, we specialise to linear spaces $L \subset \P^n$ and investigate 
their coordinate-wise powers $L^{\circ r}$. First, we highlight the dependence 
of $L^{\circ r}$ on the geometry of a finite point configuration associated to 
$L \subset \P^n$. For $r=2$, we point out its relation to symmetric matrices with degenerate eigenvalues. Based on this, we classify the coordinate-wise squares of 
lines and planes. Finally, we turn to the case of squaring linear spaces in high-dimensional ambient space.

\bigskip
\subsection{Point configurations}

We study the defining ideal of $L^{\circ r}$ for a $k$-dimensional
linear space $L \subset \P^n$. The degrees of its minimal generators do not 
change under rescaling and permuting coordinates of $\P^n$, i.e.\ under the 
actions of the algebraic torus $\mathbb{G}_m^{n+1} = (\C^*)^{n+1}$ and the 
symmetric group  $\mathfrak{S}_{n+1}$. Fixing a $(k+1)$-dimensional vector 
space $W$, we have the identification
\begin{equation*}
\begin{aligned}
\{\text{\small orbits of $\Gr(k,\P^n)$ under $\mathbb{G}_m^{n+1} 
  \rtimes 
  \mathfrak{S}_{n+1}$}\} &\leftrightarrow \left\{\substack{\text{\small 
    finite multi-sets $Z \subset \P W^*$ with $\langle Z \rangle = \P W^*$} \\ 
  \text{\small of cardinality $\leq n+1$ up to 
    $\Aut(\P W^*)$}}\right\} \\
L = \im(\P W \xhookrightarrow{[\ell_0:\ell_1:\ldots:\ell_s:0:\ldots:0]} 
\P^n) \qquad & \substack{\mapsfrom \\ \mapsto} \qquad Z = \{[\ell_0], 
[\ell_1], \ldots, [\ell_s]\} \subset \P W^*, s \leq n.
\end{aligned}
\end{equation*}
Hence, we may express coordinate-wise powers of a linear space $L$ in terms of 
the 
corresponding finite multi-set $Z \subset \P W^*$. In fact, it is easy to check 
that the degrees of the minimal generators of the defining ideal only depend on 
the underlying set $Z$, forgetting repetitions in the multi-set. We 
study coordinate-wise powers of a linear space in terms of the 
corresponding 
non-degenerate finite point configuration.
\medskip

For the entirety of Section~\ref{sec:LinSpaces}, we establish the following 
notation: Let $L\subset \P^n$ be a linear space of dimension $k$. We understand 
$L$ as the image of a chosen linear embedding 
$\iota\colon \P W \xhookrightarrow{[\ell_0:\ldots:\ell_n]} \P^n$, where 
$W$ is a $(k+1)$-dimensional vector space and $\ell_0, \ldots, \ell_n \in W^*$ 
are linear forms defining $\iota$. Consider the finite set of points $Z 
\subset \P W^*$ given by
\[Z := \{[\ell_i] \in \P W^* \mid 0 \leq i \leq n \text{ such that } \ell_i 
\neq 0\}.\] 

Since $\ell_0, \ell_1, \ldots, \ell_n \in W^*$ define the linear embedding 
$\iota$, they cannot have a common zero in $W$. Hence, the linear span of $Z$ 
is the whole space $\P W^*$. We denote by $I(Z) \subset \Sym^\bullet W$ the 
defining ideal of $Z \subset \P W^*$. The subspace of degree~$r$ forms 
vanishing on $Z$ is written as $I(Z)_r \subset \Sym^r W$.

The main technical tool is the following observation that 
%the coordinate-wise $r$-th power 
$L^{\circ r} \subset \P^n$ equals (up to a linear re-embedding) 
the image of the $r$-th Veronese variety $\nu_r(\P W)\subset \P \Sym^r W$ under 
the projection from the linear space $\P(I(Z)_r) \subset \P \Sym^r W$. 
%consisting of the degree~$r$ forms vanishing on the finite set of points $Z$.

\begin{lemma} \label{lem:intrinsicDescription}
  The diagram
  \[\begin{tikzcd}
      & \P W \arrow[hookrightarrow]{r}{\nu_r} 
      \arrow{dr}{\psi} \arrow{d}{\varphi_r \circ \iota}
      & \P \Sym^r W \arrow[dashed]{d}{\pi} \\
      \P^n \arrow[hookleftarrow]{ur}{\iota}  \arrow{r}{\varphi_r} & \P^n 
      \arrow[hookleftarrow]{r}{\vartheta} & 
      \P(\Sym^r W/I(Z)_r)
    \end{tikzcd}\]
  commutes, where $\nu_r$ is the $r$-th Veronese embedding, $\pi$ is the linear 
  projection of $\P \Sym^r W$ from the linear space $\P(I(Z)_r)$, $\psi$ is a 
  morphism and $\vartheta$ is a linear embedding.
\end{lemma}

%\cref{lem:intrinsicDescription} tells us that describing the 
%coordinate-wise $r$-th power $L^{\circ r} = \varphi_r(L) \subset \P^n$ is 
%equivalent (up to a 
%linear re-embedding) to understanding the image of $\psi \colon \P W \to 
%\P(\Sym^r W/I(Z)_r)$ for the finite set of points $Z \subset \P W^*$.

\begin{proof}
  We observe that the morphism $\varphi_r \circ \iota$ is given by
  \[\varphi_r \circ \iota \colon \P W \to \P^n, \qquad [v] \mapsto [\ell_0^r(v) 
  : \ell_1^r(v) \ldots : \ell_n^r(v)].\]
  The $n+1$ elements $\ell_i^r \in \Sym^r W^*$ correspond to a linear map 
  $\chi \colon \Sym^r W \to \C^{n+1}$ via the natural identification 
  $(\Sym^r W^*)^{n+1} = \Hom_\C(\Sym^r W, \C^{n+1})$.
  
  The rational map $\bar{\chi}$ between projective spaces corresponding to the 
  linear map $\chi$ gives the following commuting diagram:
    \[\begin{tikzcd}
        \P W \arrow[hookrightarrow]{r}{\nu_r} \arrow{d}{\varphi_r\, \circ\, 
        \iota} &
        \P \Sym^r W \arrow[dashed]{d}{\pi} \\
        \P^n \arrow[hookleftarrow]{r}{\vartheta}  
        \arrow[leftarrow, dashed]{ur}{\bar{\chi}} & 
        \P(\Sym^r W/\ker \chi),
      \end{tikzcd}\]
  where $\vartheta$ is the linear embedding of projective spaces induced by 
  factoring $\chi$ over $\Sym^r W/\ker \chi$. In particular, $\nu_r(\P W) \cap 
  \P(\ker \chi) = \emptyset$, since $\varphi_r \circ \iota$ is defined 
  everywhere 
  on $\P W$. Hence, $\restr{\pi}{\nu_r(\P W)} \colon \nu_r(\P W) \to \P(\Sym^r 
  W/\ker \chi)$ is a morphism.
  
  Finally, we claim that $\ker \chi = I(Z)_r$. Once we know this, defining 
  $\psi := \restr{\pi}{\nu_r(\P W)} \circ \nu_r$ completes the claimed diagram.
  
  Let $f \in \Sym^r W$ such that $f\in I(Z)_r$. Naturally identifying $W$ and $W^{**}$, we may view 
  $f$ as a form of degree $r$ on $W^*$. Then, the 
  condition that $f \in I(Z)_r$ translates to $f(\ell_i) = 0 \ \forall i$. Viewing 
  $f$ as a symmetric $r$-linear form $W^* 
  \times \ldots \times W^* \to \C$, we have  $f(\ell_i, 
  \ldots, \ell_i) = 0 \ \forall i$. Also, when $f$ is considered as a linear form on $\Sym^r W^*$, $f(\ell_i^r) = 0 \ \forall i$.  
%  (viewing $f$ as a form of degree $r$),  (viewing $f$ as an $r$-linear form), i.e.\ 
%   (viewing $f$ as a linear form on $\Sym^r W^*$). 
  The latter expression is equivalent to $f \in \ker \chi$, via the 
  identification of $W$ and $W^{**}$. We conclude $I(Z)_r = \ker \chi$.
\end{proof}

In particular, we deduce the following:

\begin{prop} \label{prop:NoQuadrics}
  Let $L$ be a linear space such that the finite set of points $Z$ 
  does not lie on a degree~$r$ hypersurface. Then the 
  ideal of $L^{\circ r}$ is generated by linear and quadratic forms.
\end{prop}

%\begin{cor} \label{prop:NoQuadrics}
%  Let $L \subset \P^n$ be a linear space of dimension $k$. If the 
%corresponding 
%  finite set of points $Z$ does not lie on a degree~$r$ hypersurface, then the 
%  ideal of $L^{\circ r}$ is generated by linear and quadratic forms.
%\end{cor}

\begin{proof}
  Since $I(Z)_r = 0$, we deduce from 
  \cref{lem:intrinsicDescription} that $L^{\circ r} = \varphi_r(L)$ 
  is a linear re-embedding of the $k$-dimensional $r$-th Veronese variety 
  $\nu_r(\P W) \subset \P \Sym^r W$. The ideal of this Veronese variety is 
  generated by quadrics. Since $\dim \Sym^r W = \tbinom{k+r}{r}$, 
  the 
  linear 
  re-embedding $\vartheta \colon \P \Sym^r W \hookrightarrow \P^n$ adds 
  $n-\smash{\binom{k+r}{r}}+1$ linear forms to the ideal.
\end{proof}

\bigskip
\subsection{Degenerate eigenvalues and squaring}

We now specialise to the case of coordinate-wise squaring, i.e.\ $r = 2$. This 
case has special geometric importance, since it corresponds to computing 
the image of a linear space under the quotient of $\P^n$ by the 
reflection group generated by the coordinate hyperplanes. 
In this section 
through \cref{prop:eigenvaluesEarly} we point out that the case of coordinate-wise square of a linear space is closely related to studying symmetric matrices with 
a degenerate 
spectrum of eigenvalues. Here, we 
interpret $\P \Sym^2 \F^{k+1}$ (for $\F = \R$ or $\C$) as the projective space 
consisting of symmetric $(k+1) \times (k+1)$-matrices up to scaling with entries in 
$\F$.
%We will see that this case is closely 
%related to studying symmetric matrices with a very degenerate spectrum of 
%eigenvalues, see \cref{cor:eigenvalues}.

\begin{prop} \label{prop:eigenvaluesEarly}
  %  Let $L \subset \P^n$ be any $k$-dimensional linear space such that its 
  %  point configuration $Z$ lies on a unique and smooth quadric. Let $X \subset \P \Sym^2 \R^{k+1}$ be the set of real symmetric $(k+1)\times 
  %  (k+1)$-matrices with an eigenvalue of multiplicity $\geq k$. 
  %  Then the Zariski closure of $X$ in $\P \Sym^2 \C^{k+1}$ is projectively equivalent to the projective cone over $L^{\circ 2}$. 
  Let $X \subset \P \Sym^2 \R^{k+1}$ be the set of real symmetric $(k+1)\times 
  (k+1)$-matrices with an eigenvalue of multiplicity $\geq k$. 
  Then the Zariski closure of $X$ in $\P \Sym^2 \C^{k+1}$ is projectively equivalent to the projective cone over the coordinate-wise square $L^{\circ 2}$ of any $k$-dimensional linear space $L$ whose point configuration $Z \subseteq \P W^*$ lies on a unique and smooth quadric.
\end{prop}

\begin{proof}
  Let $L \subset \P^n$ be a $k$-dimensional linear space such that 
  $I(Z)_2$ is spanned by a smooth quadric $q \in \P \Sym^2 W$. Choosing 
  coordinates of $W \cong \C^{k+1}$, we identify points in $\P \Sym^2 W$ with 
  complex symmetric $(k+1) \times (k+1)$-matrices up to scaling and we can 
  assume $q = \id \in \P \Sym^2 W$. The second Veronese variety 
  $\nu_2(\P W) \subset \P \Sym^2 W$ consists of rank~$1$ matrices. Let $X_0
  \subset \P(\Sym^2 W/\langle q \rangle)$ be the image of $\nu_2(\P W)$ under 
  the natural projection. By \cref{lem:intrinsicDescription},
  $X_0$ is the coordinate-wise square $L^{\circ 2}$ up to a linear re-embedding.
  
  The projective cone over $X_0 \cong L^{\circ 2}$ is the subvariety $X_1 
  \subset \P \Sym^2 W$ 
  consisting of complex symmetric matrices 
  $M$ such that the set $M+\langle \id \rangle$ contains a matrix of 
  rank~$\leq 1$. We observe that the rank of $M-\lambda \id$ is the codimension 
  of the eigenspace of $M$ with respect to $\lambda \in \C$. Hence,
  \[X_1 = \{M \in \P \Sym^2 \C^{k+1} \mid M \text{ has an eigenspace of 
    codimension~$\leq 1$}\}.\]
  
  We are left to show that $X_1$ is the Zariski closure in $\P \Sym^2 \C^{k+1}$ 
  of 
  $X \subset \P \Sym^2 \R^{k+1}$. Since 
  real symmetric matrices are diagonalizable, the multiplicity of an 
  eigenvalue 
  is the dimension of the corresponding eigenspace. Hence, $X_1 \cap \P \Sym^2 
  \R^{k+1}= X$. The set $X$ is the orbit of the line $V := \{\diag(\lambda, 
  \ldots, \lambda, \mu) \mid [\lambda:\mu] \in \P_\R^1\}$ under the action of 
  $O(k+1).$ The action is given by 
  conjugation 
  with orthogonal matrices and the stabiliser is $O(k) \times \{\pm 1\}$. 
  Therefore, $X$ has real dimension $\dim V + \dim O(k+1) - \dim O(k) = k+1$. 
  Also, $X_1$ is the projective cone over $X_0 \cong L^{\circ 2}$, so it 
  is a ${(k+1)}$-dimensional irreducible complex variety. We conclude that 
  $X_1$ is 
  the Zariski closure of $X$ in $\P \Sym^2 \C^{k+1}$.
\end{proof}

We illustrate \cref{prop:eigenvaluesEarly} in the case of $3 \times 3$-matrices:

\begin{ex}
  Consider the set of real symmetric $3 \times 3$-matrices with a repeated eigenvalue. We denote its Zariski closure in $\P \Sym^3 \C^2$ by $Y$. By \cref{prop:eigenvaluesEarly}, it can be understood in terms of the coordinate-wise square $L^{\circ 2}$ for some plane $L$. We make this explicit as follows: Consider the planar point configuration
  \[Z = \{[1:i:0], [1:-i:0], [1:0:i], [1:0:-i], [0:1:i]\} \subseteq \P^2,\]
  lying only on the conic $V(x^2+y^2+z^2)$. Let $L$ be the corresponding plane in $\P^4$, given as the image of
  %    \[\P^2 \hookrightarrow \P^4, \qquad [x:y:z] \mapsto [x+z:x-z:y+z:y-z:3x+4y+5z].\]
  \[\iota \colon \P^2 \hookrightarrow \P^4, \qquad [x:y:z] \mapsto [x+iy:x-iy:x+iz:x-iz:y+iz].\]
  Under the linear embedding
  \begin{align*}
    \psi \colon \P^4 &\hookrightarrow \P \Sym^2 \C^3, \\ 
    {\scriptsize [a:b:c:d:e]} &\mapsto {\scriptsize
      % a = x^2+2ixy-y^2, b = x^2-2ixy-y^2, c = x^2+2ixz-z^2, d = x^2-2ixz-z^2, e = y^2+2iyz-z^2
      % x^2-y^2 = (a+b)/2, x^2-z^2 = (c+d)/2, xy = (a-b)/4i, xz = (c-d)/4i, yz = (a+b-c-d+2e)/4i
      % 2x^2-y^2-z^2 = (a+b+c+d)/2,
      % 2y^2-x^2-z^2 = (-2a-2b+c+d)/2
      % 2z^2-x^2-y^2 = (-2c-2d+a+b)/2
      % rescaled everything by 12.
      \begin{bmatrix}
        2(a+b+c+d) & 3i(-a+b) & 3i(-c+d) \\
        3i(-a+b) & 6(-2a-2b+c+d) & 3i(-a-b+c+d-2e) \\
        3i(-c+d) & 3i(-a-b+c+d-2e) & 6(a+b-2c-2d),
    \end{bmatrix}},
  \end{align*}
  the plane $L$ gets mapped into $Y$. Indeed, it is easily checked that a point $[x:y:z]$ gets mapped to the matrix $-4(x^2+y^2+z^2)\id+12(x,y,z)^T(x,y,z)$ under the composition $\psi \circ \iota \colon \P^2 \to \P \Sym^2 \C^3$; note that this matrix has a repeated eigenvalue. More precisely, \cref{prop:eigenvaluesEarly} shows that $Y$ is the projective cone over $\psi(L^{\circ 2})$ with the vertex $\id$.
\end{ex}

In \cref{ssec:squaringHighDim} we give an explicit set-theoretic description of the coordinate-wise square of a linear space in high-dimensional ambient space. We will show the following result as a special case of \cref{thm:uniqueQuadricCase}. Given a matrix $A \in \C^{s \times s}$, we denote a $2\times 2$ minor of $A$ by $A_{ij|k\ell}$ where $i,j$ are the rows and $k,\ell$ are the columns of the minor.

\begin{cor} \label{cor:eigenvalues}
  Let $s \geq 4$. A symmetric matrix $A \in \C^{s \times s}$ has an eigenspace 
  of codimension~$\leq 1$ if and only if its $2\times 2$-minors
  %$A_{ij|k \ell} \in \C$ 
  satisfy the following for $i,j,k,\ell \leq s$ distinct:
  \[A_{ij|k\ell} = 0, \qquad A_{ik|i\ell} = A_{jk|j\ell} \qquad \text{and} 
  \qquad A_{ik|ik} - A_{i\ell|i\ell} = A_{jk|jk} - A_{j \ell|j \ell}.\]
  These equations describe the Zariski closure in the complex vector space 
  $\Sym^2 \C^s$ of the set of real symmetric matrices with an eigenvalue 
  of multiplicity $\geq s-1$.
\end{cor}

\bigskip
\subsection{Squaring lines and planes}

%Similar ideas as in this section apply to coordinate-wise $r$-th powers.
 
In this subsection we consider the low-dimensional cases and classify the coordinate-wise 
squares of lines and planes in arbitrary ambient spaces.

\begin{thm}[Squaring lines] \label{thm:squaringLines}
  Let $L$ be a line in $\P^n$. 
%  The 
%  coordinate-wise square $L^{\circ 2} \subset \P^n$ depends on the geometry of 
%  the finite set of points $Z \subset \P W^*$ as follows:
  \begin{enumerate}[(i)]
    \item If $|Z| = 2$, then $L^{\circ 2}$ is a line in $\P^n$.
    \item If $|Z| > 2$, then $L^{\circ 2}$ is a smooth conic in $\P^n$.
    %contained in a 
    %plane inside $\P^n$.
    %complete intersection of $(n-2)$ hyperplanes and one quadric.
  \end{enumerate}
\end{thm}

\begin{proof}
%  As before, let $W$ be a $2$-dimensional vector space and let $\ell_0, 
%\ell_1, 
%  \ldots, \ell_n \in W^*$ be linear forms defining a linear embedding $\iota 
%  \colon \P W \hookrightarrow \P^n$ whose image is the line $L \subset \P^n$. 
%  Let $Z \subset \P W^*$ be as before.
  Since $Z \subset \P W^*$ spans the projective line $\P W^*$, we must have 
  $|Z| \geq 2$.
  
  If $|Z| > 2$, then $I(Z)_2 = 0$, since no non-zero quadratic form on the 
  projective line $\P W^*$ vanishes on all points of $Z$. Then 
  \cref{lem:intrinsicDescription} implies that $L^{\circ 2} = (\varphi_2 \circ 
  \iota)(\P W) $ is a linear re-embedding of $\nu_2(\P W) $, which is a smooth 
  conic in the plane $\P \Sym^2 W \cong \P^2$.
  
  If $|Z| = 2$, then $\dim I(Z)_2 = 1$, since up to scaling there is a unique 
  quadric vanishing on the points $Z $. By 
  \cref{lem:intrinsicDescription}, the image $\varphi_2(L)$ lies in a 
  projective line $\P^1 \cong \vartheta(\P(\Sym^2 W/I(Z)_2)) \subset \P^n$. On 
  the other hand $\dim L^{\circ 2} = \dim L = 1$. Hence, $L^{\circ 2} = 
  \varphi_2(L)$ is a line 
  in $\P^n$.
\end{proof}

\begin{remark} \label{rem:lineMatroidInvariant}
  We observe that the two possibilities in \cref{thm:squaringLines} 
  for the coordinate-wise square of a line $L$ differ in degree. 
  In particular, \cref{cor:DegreeMatroidInvariant} shows that it only depends 
  on the linear matroid $\mathcal M_L$ whether $L^{\circ 2} $ is a 
  line or a (re-embedded) plane conic.
\end{remark}

\begin{remark}
  In the Grassmannian of lines $\Gr(1, \P^n)$, consider the locus $\Gamma 
  \subset \Gr(1, \P^n)$ of those lines $L $ whose coordinate-wise 
  square $L^{\circ 2}$ is a line. 
  Considering Plücker coordinates $p_{ij}$ on the Grassmannian $\Gr(1, \P^n)$, 
  we observe that $\Gamma$ is the subvariety of $\Gr(1, \P^n)$ given by the 
  vanishing of $p_{ij} p_{jk} p_{ki}$ for all $i,j,k \in \{0,1, \ldots, n\}$ 
  distinct:
    \[\Gamma = V(p_{ij} p_{jk} p_{ki} \mid i,j,k \in \{0,1, \ldots, n\} \text{ 
    distinct}) \subset \Gr(1, \P^n).\]
  Indeed, if $L$ is the image of an embedding $\P^1 \xhookrightarrow{B} 
  \P^n$ given by a chosen rank~$2$ matrix $B \in \C^{(n+1)\times 2}$, then $Z 
  \subset (\P^1)^*$ is the set of points corresponding to the non-zero rows of 
  $B$. Then $|Z|=2$ if and only if among any three distinct rows of $B$ there 
  always exist two linearly dependent rows. In terms of the Plücker 
  coordinates, which are given by the $2 \times 2$-minors of $B$, this 
  translates into the vanishing condition above.
\end{remark}

%Case (i) occurs if and only if there exists $I \subset \{0,1,\ldots,n\}$ of 
%cardinality $|I|=n-2$ such that $I \cup \{i\}$ is a basis of the matroid 
%$\mathcal M_L$ for all $i \in \{0,1,\ldots,n\} \setminus I$.

\begin{thm}[Squaring planes] \label{thm:squaringPlanes}
  Let $L$ be a plane in $\P^n$. The defining ideal $I \subset \Cx$ of $L^{\circ 
  2}$ depends on the geometry of the planar configuration of $Z\subset \P W^*$ as follows (see 
  Figure~\ref{fig:PointConfiguration}):
  \begin{enumerate}[(i)]
    \item If $Z$ is not contained in any conic, then $I$ is 
    minimally generated by $n-5$ linear forms and 6~quadratic forms.
    \item If $Z$ is contained in a unique conic $Q \subset \P W^*$, we 
    distinguish two cases:
          \begin{enumerate}[(a)]
            \item If $Q$ is irreducible, then $I$ is minimally 
            generated by 
            $n-4$ linear forms and 7 cubic forms.
            \item If $Q$ is reducible, then $L^{\circ 2}$ is the complete 
            intersection of $n-4$ hyperplanes and 2~quadrics.
          \end{enumerate}
    \item If $Z$ is contained in several conics, we distinguish three cases:
          \begin{enumerate}[(a)]
            \item If $|Z| = 3$, then $I$ is minimally generated by 
            $n-2$ linear forms.
            \item If $|Z| = 4$ and no three points of $Z$ are collinear, then 
            $I$ is minimally generated by $n-3$ linear forms and 
            one quartic form.
            \item If $|Z| \geq 3$ and all but one of the points of $Z$ lie on a line,
            %If $Z$ contains at least three collinear points,
            then $I$ is minimally generated by $n-3$ linear forms and 
            one quadratic form.
          \end{enumerate}
  \end{enumerate}
\end{thm}

%\begin{figure}[h]
%\hspace{-0.3cm}\begin{tabular}{c@{\hskip 0.6cm}cc@{\hskip 0.6cm}ccc}
%  \hspace{-0.5cm} {\input{imgNoQuadric.tex}} \hspace{-0.5cm} & \hspace{-0.5cm} 
%  {\input{imgIrredQuadric.tex}} \hspace{-0.5cm} & 
%  \hspace{-0.5cm} {\input{imgRedQuadric.tex}} \hspace{-0.5cm} & \hspace{-0.5cm} 
%  {\input{img3Pts.tex}} \hspace{-0.5cm} & 
%  \hspace{-0.5cm} {\input{img4Pts.tex}} \hspace{-0.5cm} & \hspace{-0.5cm} 
%  {\input{imgLinePlusPoint.tex}} \hspace{-0.5cm} \\[-0.2cm]
%  \hspace{-0.5cm} \footnotesize (i) & \hspace{-0.5cm} \footnotesize (ii).(a) & 
%  \hspace{-0.5cm} \footnotesize 
%  (ii).(b) & 
%  \hspace{-0.5cm} \footnotesize (iii).(a) & \hspace{-0.5cm} \footnotesize 
%  (iii).(b) & \hspace{-0.5cm} \footnotesize (iii).(c)
%\end{tabular}
%  \caption{Dependence of $L^{\circ 2}$ on the planar point configuration $Z$}
%  \label{fig:PointConfiguration}
%\end{figure}

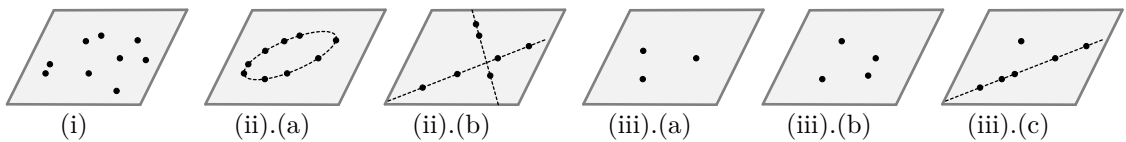
\begin{figure}[h]
	\hspace{-0.3cm}\begin{tabular}{c@{\hskip 0.6cm}cc@{\hskip 0.6cm}ccc}
		\hspace{-0.5cm} {\definecolor{yqyqyq}{rgb}{0.5019607843137255,0.5019607843137255,0.5019607843137255}
			\definecolor{dgr}{rgb}{0.25,0.25,0.25}
			\begin{tikzpicture}[line cap=round,line join=round,>=triangle 
			45,x=1.0cm,y=1.0cm,scale=0.2]
			\clip(-4.5,-2) rectangle (7.8,5);
			\fill[line width=1.pt,color=yqyqyq,fill=yqyqyq,fill 
			opacity=0.10000000149011612] (-1.2608503781345857,4.700842335843258) -- 
			(-4.345209896276158,-1.528953126541707) -- 
			(4.419257249334061,-1.5594913395926135) -- 
			(7.50361676748,4.7313805488941645) -- cycle;
			\draw [line width=1.pt,color=yqyqyq] (-1.2608503781345857,4.700842335843258)-- 
			(-4.345209896276158,-1.528953126541707);
			\draw [line width=1.pt,color=yqyqyq] (-4.345209896276158,-1.528953126541707)-- 
			(4.419257249334061,-1.5594913395926135);
			\draw [line width=1.pt,color=yqyqyq] (4.419257249334061,-1.5594913395926135)-- 
			(7.50361676748,4.7313805488941645);
			\draw [line width=1.pt,color=yqyqyq] (7.50361676748,4.7313805488941645)-- 
			(-1.2608503781345857,4.700842335843258);
			\begin{scriptsize}
			\draw [fill=black] (1.0295156006834159,0.5152960133926188) circle (5pt);
			\draw [fill=black] (3.08,1.52) circle (5pt);
			\draw [fill=black] (4.785715805944938,1.400904191868913) circle (5pt);
			\draw [fill=black] (1.8540473530578965,3.0194294835669675) circle (5pt);
			\draw [fill=black] (-1.810538213050906,0.5152960133926188) circle (5pt);
			\draw [fill=black] (2.861808383737817,-0.6451560825418355) circle (5pt);
			%\draw[color=dgr] (-3.25,3) node {$(\mathbb{P}^2)^*$};
			\draw [fill=black] (4.247836976757553,2.7093004463893333) circle (5pt);
			\draw [fill=black] (0.8430402090918784,2.6504732895597267) circle (5pt);
			\draw [fill=black] (-1.5240915019634127,1.123094162197016) circle (5pt);
			\end{scriptsize}
			\end{tikzpicture}} \hspace{-0.5cm} & \hspace{-0.5cm} 
		{\definecolor{yqyqyq}{rgb}{0.5019607843137255,0.5019607843137255,0.5019607843137255}
			\definecolor{dgr}{rgb}{0.25,0.25,0.25}
			\begin{tikzpicture}[line cap=round,line join=round,>=triangle 
			45,x=1cm,y=1cm,scale=0.2]
			\clip(-4.5,-2) rectangle (7.8,5);
			\fill[line width=1.pt,color=yqyqyq,fill=yqyqyq,fill 
			opacity=0.10000000149011612] (-1.2608503781345857,4.700842335843258) -- 
			(-4.345209896276158,-1.528953126541707) -- 
			(4.419257249334061,-1.5594913395926135) -- 
			(7.50361676748,4.7313805488941645) -- cycle;
			\draw [line width=1.pt,color=yqyqyq] (-1.2608503781345857,4.700842335843258)-- 
			(-4.345209896276158,-1.528953126541707);
			\draw [line width=1.pt,color=yqyqyq] (-4.345209896276158,-1.528953126541707)-- 
			(4.419257249334061,-1.5594913395926135);
			\draw [line width=1.pt,color=yqyqyq] (4.419257249334061,-1.5594913395926135)-- 
			(7.50361676748,4.7313805488941645);
			\draw [line width=1.pt,color=yqyqyq] (7.50361676748,4.7313805488941645)-- 
			(-1.2608503781345857,4.700842335843258);
			\draw [dash pattern=on 1pt off 1pt, rotate 
			around={-156.8071774809104:(1.2292054707106181,1.697331258300542)},line 
			width=0.5pt] (1.2292054707106181,1.697331258300542) ellipse 
			(3.2786335829076863cm and 1.0282157046040252cm);
			\begin{scriptsize}
			\draw [fill=black] (1.0295156006834159,0.5152960133926188) circle (5pt);
			\draw [fill=black] (3.08,1.52) circle (5pt);
			\draw [fill=black] (1.8540473530578965,3.0194294835669675) circle (5pt);
			\draw [fill=black] (-0.40578041270919835,2.0116684528870468) circle (5pt);
			\draw [fill=black] (-1.810538213050906,0.5152960133926188) circle (5pt);
			%\draw[color=dgr] (-3.25,3) node {\tiny $(ii).(a)$};
			\draw [fill=black] (-0.4357499009679917,0.13785006719710813) circle (5pt);
			\draw [fill=black] (4.247836976757553,2.7093004463893333) circle (5pt);
			\draw [fill=black] (0.8430402090918784,2.6504732895597267) circle (5pt);
			\draw [fill=black] (-1.5240915019634118,1.123094162197017) circle (5pt);
			\end{scriptsize}
			\end{tikzpicture}} \hspace{-0.5cm} & 
		\hspace{-0.5cm} {\definecolor{yqyqyq}{rgb}{0.5019607843137255,0.5019607843137255,0.5019607843137255}
			\definecolor{dgr}{rgb}{0.25,0.25,0.25}
			\begin{tikzpicture}[line cap=round,line join=round,>=triangle 
			45,x=1cm,y=1cm,scale=0.2]
			\clip(-4.5,-2) rectangle (7.8,5);
			\fill[line width=1.pt,color=yqyqyq,fill=yqyqyq,fill 
			opacity=0.10000000149011612] (-1.2608503781345857,4.700842335843258) -- 
			(-4.345209896276158,-1.528953126541707) -- 
			(4.419257249334061,-1.5594913395926135) -- 
			(7.50361676748,4.7313805488941645) -- cycle;
			\draw [line width=1.pt,color=yqyqyq] (-1.2608503781345857,4.700842335843258)-- 
			(-4.345209896276158,-1.528953126541707);
			\draw [line width=1.pt,color=yqyqyq] (-4.345209896276158,-1.528953126541707)-- 
			(4.419257249334061,-1.5594913395926135);
			\draw [line width=1.pt,color=yqyqyq] (4.419257249334061,-1.5594913395926135)-- 
			(7.50361676748,4.7313805488941645);
			\draw [line width=1.pt,color=yqyqyq] (7.50361676748,4.7313805488941645)-- 
			(-1.2608503781345857,4.700842335843258);
			\draw [line width=0.5pt,dash pattern=on 1pt off 1pt] 
			(1.3889599378045718,4.710656448124513)-- 
			(3.1119979790807197,-1.554936429243299);
			\draw [line width=0.5pt,dash pattern=on 1pt off 1pt] 
			(6.17037763999702,2.7349229517951223)-- 
			(-4.265369750415246,-1.3676918418325394);
			\begin{scriptsize}
			\draw [fill=black] (3.08,1.52) circle (5pt);
			\draw [fill=black] (1.8540473530578965,3.0194294835669675) circle (5pt);
			%\draw[color=dgr] (-3.25,3) node {\tiny $\P W^*$};
			%\draw[color=black] (1,2) node {$Z$};
			\draw [fill=black] (0.4165053550709401,0.47289794547218034) circle (5pt);
			\draw [fill=black] (-1.866046034530464,-0.42444354314966637) circle (5pt);
			\draw [fill=black] (5.126600968223972,2.324582086435178) circle (5pt);
			\draw [fill=black] (2.5826748709118235,0.3698748731890493) circle (5pt);
			\draw [fill=black] (1.6456283614483025,3.7773167257836735) circle (5pt);
			\end{scriptsize}
			\end{tikzpicture}} \hspace{-0.5cm} & \hspace{-0.5cm} 
		{\definecolor{yqyqyq}{rgb}{0.5019607843137255,0.5019607843137255,0.5019607843137255}
			\definecolor{dgr}{rgb}{0.25,0.25,0.25}
			\begin{tikzpicture}[line cap=round,line join=round,>=triangle 
			45,x=1.0cm,y=1.0cm,scale=0.2]
			\clip(-4.5,-2) rectangle (7.8,5);
			\fill[line width=1.pt,color=yqyqyq,fill=yqyqyq,fill 
			opacity=0.10000000149011612] (-1.2608503781345857,4.700842335843258) -- 
			(-4.345209896276158,-1.528953126541707) -- 
			(4.419257249334061,-1.5594913395926135) -- 
			(7.50361676748,4.7313805488941645) -- cycle;
			\draw [line width=1.pt,color=yqyqyq] (-1.2608503781345857,4.700842335843258)-- 
			(-4.345209896276158,-1.528953126541707);
			\draw [line width=1.pt,color=yqyqyq] (-4.345209896276158,-1.528953126541707)-- 
			(4.419257249334061,-1.5594913395926135);
			\draw [line width=1.pt,color=yqyqyq] (4.419257249334061,-1.5594913395926135)-- 
			(7.50361676748,4.7313805488941645);
			\draw [line width=1.pt,color=yqyqyq] (7.50361676748,4.7313805488941645)-- 
			(-1.2608503781345857,4.700842335843258);
			\begin{scriptsize}
			\draw [fill=black] (3.08,1.52) circle (5pt);
			\draw [fill=black] (-0.40578041270919835,2.0116684528870468) circle (5pt);
			%\draw[color=dgr] (-3.25,3) node {$(\mathbb{P}^2)^*$};
			\draw [fill=black] (-0.4357499009679917,0.13785006719710813) circle (5pt);
			\end{scriptsize}
			\end{tikzpicture}} \hspace{-0.5cm} & 
		\hspace{-0.5cm} {\definecolor{yqyqyq}{rgb}{0.5019607843137255,0.5019607843137255,0.5019607843137255}
			\definecolor{dgr}{rgb}{0.25,0.25,0.25}
			\begin{tikzpicture}[line cap=round,line join=round,>=triangle 
			45,x=1.0cm,y=1.0cm,scale=0.2]
			\clip(-4.5,-2) rectangle (7.8,5);
			\fill[line width=1.pt,color=yqyqyq,fill=yqyqyq,fill 
			opacity=0.10000000149011612] (-1.2608503781345857,4.700842335843258) -- 
			(-4.345209896276158,-1.528953126541707) -- 
			(4.419257249334061,-1.5594913395926135) -- 
			(7.50361676748,4.7313805488941645) -- cycle;
			\draw [line width=1.pt,color=yqyqyq] (-1.2608503781345857,4.700842335843258)-- 
			(-4.345209896276158,-1.528953126541707);
			\draw [line width=1.pt,color=yqyqyq] (-4.345209896276158,-1.528953126541707)-- 
			(4.419257249334061,-1.5594913395926135);
			\draw [line width=1.pt,color=yqyqyq] (4.419257249334061,-1.5594913395926135)-- 
			(7.50361676748,4.7313805488941645);
			\draw [line width=1.pt,color=yqyqyq] (7.50361676748,4.7313805488941645)-- 
			(-1.2608503781345857,4.700842335843258);
			\begin{scriptsize}
			\draw [fill=black] (3.08,1.52) circle (5pt);
			%\draw[color=dgr] (-3.25,3) node {$(\mathbb{P}^2)^*$};
			\draw [fill=black] (-0.4357499009679917,0.13785006719710813) circle (5pt);
			\draw [fill=black] (0.8430402090918784,2.6504732895597267) circle (5pt);
			\draw [fill=black] (2.5826748709118235,0.3698748731890493) circle (5pt);
			\end{scriptsize}
			\end{tikzpicture}} \hspace{-0.5cm} & \hspace{-0.5cm} 
		{\definecolor{yqyqyq}{rgb}{0.5019607843137255,0.5019607843137255,0.5019607843137255}
			\definecolor{dgr}{rgb}{0.25,0.25,0.25}
			\begin{tikzpicture}[line cap=round,line join=round,>=triangle 
			45,x=1.0cm,y=1.0cm,scale=0.2]
			\clip(-4.5,-2) rectangle (7.8,5);
			\fill[line width=1.pt,color=yqyqyq,fill=yqyqyq,fill 
			opacity=0.10000000149011612] (-1.2608503781345857,4.700842335843258) -- 
			(-4.345209896276158,-1.528953126541707) -- 
			(4.419257249334061,-1.5594913395926135) -- 
			(7.50361676748,4.7313805488941645) -- cycle;
			\draw [line width=1.pt,color=yqyqyq] (-1.2608503781345857,4.700842335843258)-- 
			(-4.345209896276158,-1.528953126541707);
			\draw [line width=1.pt,color=yqyqyq] (-4.345209896276158,-1.528953126541707)-- 
			(4.419257249334061,-1.5594913395926135);
			\draw [line width=1.pt,color=yqyqyq] (4.419257249334061,-1.5594913395926135)-- 
			(7.50361676748,4.7313805488941645);
			\draw [line width=1.pt,color=yqyqyq] (7.50361676748,4.7313805488941645)-- 
			(-1.2608503781345857,4.700842335843258);
			\draw [line width=0.5pt,dash pattern=on 1pt off 1pt] 
			(6.17037763999702,2.7349229517951223)-- 
			(-4.265369750415246,-1.3676918418325394);
			\begin{scriptsize}
			\draw [fill=black] (3.08,1.52) circle (5pt);
			%\draw[color=dgr] (-3.25,3) node {$(\mathbb{P}^2)^*$};
			\draw [fill=black] (-0.4357499009679917,0.13785006719710813) circle (5pt);
			\draw [fill=black] (0.8430402090918784,2.6504732895597267) circle (5pt);
			\draw [fill=black] (0.4165053550709401,0.4728979454721804) circle (5pt);
			\draw [fill=black] (-1.866046034530464,-0.42444354314966637) circle (5pt);
			\draw [fill=black] (5.126600968223972,2.3245820864351776) circle (5pt);
			\end{scriptsize}
			\end{tikzpicture}} \hspace{-0.5cm} \\[-0.2cm]
		\hspace{-0.5cm} \footnotesize (i) & \hspace{-0.5cm} \footnotesize (ii).(a) & 
		\hspace{-0.5cm} \footnotesize 
		(ii).(b) & 
		\hspace{-0.5cm} \footnotesize (iii).(a) & \hspace{-0.5cm} \footnotesize 
		(iii).(b) & \hspace{-0.5cm} \footnotesize (iii).(c)
	\end{tabular}
	\caption{Dependence of $L^{\circ 2}$ on the planar point configuration $Z$}
	\label{fig:PointConfiguration}
\end{figure}

\begin{proof}
  Notice that $k=2$, so $\dim W = 3$.
  \begin{enumerate}[(i)]
    \item 
    %If $Z$ does not lie on any conic in the projective plane $\P W^*$, 
    %then $I(Z)_2 = 0$, hence $L^{\circ 2} \subset \P^n$ is by 
    If $I(Z)_2 = 0$, then $L^{\circ 2} \subset \P^n$ is by 
    \cref{lem:intrinsicDescription} a linear re-embedding of the 
    Veronese surface $\nu_2(\P W) \subset \P \Sym^2 W$. The ideal of the 
    $\nu_2(\P W)$ is minimally generated by six quadrics. 
    Indeed, choosing a basis for $W$, we may understand points in $\P \Sym^2 W$ 
    as symmetric $3\times 3$-matrices up to scaling. Then $\nu_2(\P W)$ is the 
    subvariety corresponding to symmetric rank~1 matrices, which is the vanishing set of the six quadratic polynomials corresponding to the $2 \times 2$-minors. 
    Since $\dim \P \Sym^2 W = 5$, the linear re-embedding $\P \Sym^2 W 
    \hookrightarrow \P^n$ adds $n-5$ linear forms to $I$.
    
    \item We can choose 
    %coordinates $z_0,z_1,z_2 \in W$ on the projective plane $\P W^*$ 
    a basis $\{z_0,z_1,z_2\}$ of $W$
    such that the unique reduced plane conic through $Z \subset \P W^*$ is 
    with respect to these coordinates
    given by the vanishing of either $q_1 := z_0^2-2z_1 z_2 \in \Sym^2 W$ 
    or $q_2 := z_1 z_2 \in \Sym^2 W$.
    
    We consider the basis $\{z_1^2, z_2^2, 2z_0 z_1, 2z_0 z_2, 2z_1 z_2\}$ of 
    $\Sym^2 W/\langle q_1\rangle$ and the basis $\{z_0^2, z_1^2, z_2^2, 2z_0 
    z_1, 2z_0 z_2\}$ of $\Sym^2 W/\langle q_2\rangle$.
    With respect to these choices of bases, the morphism $\psi\colon \P W \to 
    \P(\Sym^2 W/I(Z)_2)$ is given as
    \begin{align*}
      &\psi\colon \P^2 \to \P^4, \quad [a_0:a_1:a_2] \mapsto [a_1^2 : a_2^2 : 
      a_0 a_1 : a_0 a_2 : a_0^2+a_1a_2] \\
      \text{or} \qquad &\psi\colon \P^2 \to \P^4, \quad [a_0:a_1:a_2] \mapsto 
      [a_0^2 : a_1^2 : a_2^2 : a_0 a_1 : a_0 a_2].
    \end{align*}
    In the first case, we checked computationally with 
    \texttt{Macaulay2} \cite{M2} that the ideal is minimally generated by seven 
    cubics. A 
    structural description of these quadrics and cubics will be given in 
    the 
    proof of \cref{thm:uniqueQuadricCase}. The image of the second morphism is 
    a complete intersection of two binomial 
    quadrics. By 
    \cref{lem:intrinsicDescription}, the coordinate-wise 
    square $L^{\circ 2}$ 
    arises from the image of $\psi$ via a linear re-embedding $\P^4 
    \hookrightarrow \P^n$, producing additional $n-4$~linear forms in $I$.
    
    \item  In case~(a), the set $Z$ consists of three points 
    spanning the projective plane $\P W^*$, so $\dim \Sym^2 W/I(Z)_2 = 3$. 
    Then by \cref{lem:intrinsicDescription}, the coordinate-wise square 
    $L^{\circ 2}$ is contained in a plane $\P^2 \cong \vartheta(\P(\Sym^2 
    W/I(Z)_2)) 
    \subset \P^n$. On the other hand, $\dim L^{\circ 2} = \dim L = 2$, so 
    $L^{\circ 2} \subset \P^n$ must be a 
    plane in $\P^n$.
    
    For case~(b), we may assume that 
     \[Z = \{[1:0:0],[0:1:0],[0:0:1],[-1:-1:-1]\}\]
    for a suitably chosen basis $\{\ell_0,\ell_1,\ell_2\}$ of $W^*$. By 
    \cref{lem:intrinsicDescription}, $L^{\circ 2} \subset \P^n$ is a linear 
    re-embedding of the image of $\psi \colon \P W \to \P(\Sym^2 
    W/I(Z)_2)$.
    On the other hand, the plane $L' := V(x_0+x_1+x_2+x_3) \subset \P^3$ is the 
    image of
    $\P W \xhookrightarrow{[\ell_0:\ell_1:\ell_2:-\ell_0-\ell_1-\ell_2]} \P^3$, 
    so $Z$ can also be viewed as the finite set of points associated to 
    $L'$. Applying 
    \cref{lem:intrinsicDescription} 
    to $L' \subset \P^3$ shows that the image of 
    %the natural morphism
    $\psi \colon \P W \to \P(\Sym^2 W/I(Z)_2)$ is the coordinate-wise square 
    ${L'}^{\circ 2} \subset \P^3$. 
    %We have seen in \cref{ex:quarticSingularSurface} that 
    %${L'}^{\circ 2} \subset \P^3$ is a quartic surface. Then 
    %This is the quartic surface from \cref{ex:quarticSingularSurface}.     
    Hence, $L^{\circ 2} \subset \P^n$ is a linear re-embedding of the quartic 
    surface from \cref{ex:quarticSingularSurface} into higher dimension.
    
    Finally, we consider case~(c). Consider three points $p_1,p_2,p_3 \in 
    Z$ lying on a line $T \subset \P W^*$. Then $T$ must be an irreducible 
    component of each conic through $Z$. Since $Z$ spans the projective plane 
    $\P W^*$, there must also be a point $p_0 \in Z$ outside of $T$. All points 
    in $Z \setminus \{p_0\}$ must lie on the line $T$, as otherwise there could 
    be at most one conic passing through $Z$. If $Z' := \{p_0,p_1,p_2,p_3\} 
    \subset Z$, then each conic passing through $Z'$ also passes through $Z$, 
    i.e.\ $I(Z)_2 = I(Z')_2$.
    
    We may choose a basis $z_0,z_1,z_2$ of $W$ such that $Z' \subset \P W^*$ 
    with respect to these coordinates is given by
      \[Z' = \{[1:0:0],[0:1:0],[0:0:1],[0:1:1]\}.\]
    The plane $L' := V(x_1+x_2-x_3) \subset \P^3$ is the image of 
    $\P^2 \xhookrightarrow{[z_0:z_1:z_2:z_1+z_2]} \P^3$, so $Z'$ can be viewed 
    as the finite set of points associated to $L'$.
    % as in \cref{notation:finPts}.
    \cref{lem:intrinsicDescription} shows 
    %for the linear space $L'$
    that ${L'}^{\circ 2} \subset \P^3$ coincides with the 
    image of the morphism $\psi \colon \P W \to \P(\Sym^2 W/I(Z')_2)$. On the 
    other hand, 
    %for the linear space $L$,
    \cref{lem:intrinsicDescription} shows that 
    $L^{\circ 2} \subset \P^n$ is a linear re-embedding of $\P W \to \P(\Sym^2 
    W/I(Z)_2)$. From $I(Z)_2 = I(Z')_2$, we deduce that $L^{\circ 
    2} \subset \P^n$ is a linear re-embedding of
    the quadratic surface
    \[{L'}^{\circ 2} = V(x_1+x_2-x_3)^{\circ 2} = V(x_1^2+x_2^2+x_3^2-2x_1 
    x_2-2x_2 x_3 
    -2x_3x_1) \subset \P^3,\]
    %$V(x_1+x_2-x_3)^{\circ 2} \subset \P^3$, which is 
%    the smooth quadratic surface $V(x_1^2+x_2^2+x_3^2-2x_1 x_2-2x_2 x_3 
%    -2x_3x_1) \subset \P^3$
    %a smooth quadratic surface in $\P^3$
    as we compute from \cref{prop:hypersurfaces}.
    \qedhere
  \end{enumerate}
\end{proof}

%
%\begin{todo}
%  Formulate the case distinction in terms of equations in Plücker coordinates 
%  in $\Gr(2,\P^n)$.
%\end{todo}

\begin{remark} \label{rem:planeNoMatroidInvariant}
  Opposed to \cref{rem:lineMatroidInvariant}, the structure of the 
  coordinate-wise square of a plane $L \subset \P^n$ does \emph{not} only 
  depend on the linear matroid of $L$: For $n=5$, it can happen 
  both in case~(i) and case~(ii).(a) of \cref{thm:squaringPlanes} that 
  $\mathcal M_L = \{I \subset \{0,1,\ldots,5\} \mid |I| \leq 3\}$.
\end{remark}

\bigskip
\subsection{Squaring in high ambient dimensions} \label{ssec:squaringHighDim}
%By \cref{prop:NoQuadrics}, we know the description of the coordinate-wise 
%square 
%of a linear space $L 
%$ in the most generic situation, namely if the corresponding finite 
%set of points $Z $ does not lie on a quadric. In 
%\cref{thm:uniqueQuadricCase}, we investigate the next degenerate case, that $Z 
%$ lies on a unique quadric. We will see that this case is closely 
%related to studying symmetric matrices with a very degenerate spectrum of 
%eigenvalues, see \cref{cor:eigenvalues}.
%give the explicit set theoretic description of $L^{\circ 2}$ where $L$ is a linear space in a high ambient dimensional space.
Consider the case of $k$-dimensional linear spaces in $\P^n$ for 
$n \gg k$. For a \emph{general} linear space $L \in \Gr(k,\P^n)$, the 
finite set of points $Z$ does not lie on a quadric. We know from 
\cref{prop:NoQuadrics} that the coordinate-wise square $L^{\circ 2}$ is a 
linear 
re-embedding of the $k$-dimensional second Veronese variety. 
In this subsection, we investigate the first degenerate 
case where the point configuration $Z$ is a unique quadric. 

The following theorem gives the structure of coordinate-wise squares as the one appearing in 
\cref{prop:eigenvaluesEarly}. We will also prove \cref{cor:eigenvalues} by deriving the polynomials 
vanishing on the set of symmetric matrices with a comultiplicity~$1$ 
eigenvalue. \cref{prop:eigenvaluesEarly} shows that \cref{cor:eigenvalues} is a special case of the theorem stated below.

\begin{thm} \label{thm:uniqueQuadricCase}
  Let $L \subset \P^n$ be linear space of dimension $k$.
  If the point configuration $Z$ lies on a unique quadric of rank $s$, then 
  $L^{\circ 2}$ can set-theoretically be described as the vanishing set of $n-\binom{k+2}{2}+2$ 
  linear forms and
  \[\begin{cases}
      (k+3)(k+2)(k+1)(k-2)/12 \text{ quadratic forms}, &\text{if } s \geq 4, \\
      (k+3)(k+2)(k+1)(k-2)/12 \text{ quadratic and 7 cubic forms}, &\text{if }
      s = 3, \\
      (k+3)(k+2)(k+1)(k-2)/12 + 2 \text{ quadratic forms}, &\text{if } s = 2.
   \end{cases}\]
\end{thm}

In fact, for $s \geq 3$, we show that the claim holds 
\emph{scheme-theoretically}, see Remark~\ref{rem:schemeTheoretic}. We believe 
that in fact for arbitrary 
$s$ the claim is even true \emph{ideal-theoretically}.
\medskip

The remainder of this subsection is dedicated to the proof of
\cref{thm:uniqueQuadricCase}. 
%To set up the proof, we consider the following 
%situation:
It reduces to the following elimination problem. Let $k \geq 1$ and $s \geq 2$. Consider a symmetric $(k+1)\times (k+1)$-matrix 
of variables $Y := (y_{ij})_{1\leq i,j \leq k+1}$ 
and the corresponding polynomial ring
$\Cy := \C[y_{ij} ]/(y_{ij}-y_{ji}).$
Over the polynomial ring $\Cyt$, we consider the matrix $M := Y+tI_s$, where we define the matrix
\[I_s :=  \diag(\underbrace{1,\ldots,1}_{s}, \underbrace{0,\ldots,0}_{k+1-s}) 
\in \C^{(k+1)\times(k+1)}.\]

Henceforth, we denote the $2\times 2$-minors of $Y$ with rows $i\neq j$ 
and columns $\ell \neq m$ by 
  $Y_{ij|\ell m} := y_{i\ell} y_{j m} - y_{im} y_{j \ell} \in \Cy,$
and correspondingly $M_{ij|\ell m} \in \Cyt$ for the $2 \times 2$-minors of 
$M$. Let $J_0 % := (M_{ij|\ell m} \mid i \neq j, \ell \neq m)
\subset \Cyt$
denote the 
ideal generated by the $2 
\times 2$-minors of $M$. By $J := J_0 \cap \Cy$ we denote the ideal in $\Cy$ 
obtained by 
eliminating $t$ from $J_0$. 
%In the next proposition, w
We explicitly describe the elimination 
ideal $J$ for all values of $k$ and $s$. 

\begin{prop} \label{prop:matrixCompletion}
  The vanishing set 
  $V(J) \subset \P^{\binom{k+2}{2}-1}$ can
  set-theoretically be described as the zero set of
  \[\begin{cases}
  (k+3)(k+2)(k+1)(k-2)/12 \text{ quadratic forms}, &\text{if } s \geq 4, \\
  (k+3)(k+2)(k+1)(k-2)/12 \text{ quadratic and 7 cubic forms}, &\text{if }
  s = 3, \\
  (k+3)(k+2)(k+1)(k-2)/12 + 2 \text{ quadratic forms}, &\text{if } s = 2.  
  \end{cases}\]
\end{prop}

First, we observe that \cref{thm:uniqueQuadricCase} follows directly from 
\cref{prop:matrixCompletion}.

\begin{proof}[Proof of Theorem~\ref{thm:uniqueQuadricCase}]
  Analogous to the proof of \cref{prop:eigenvaluesEarly}, we identify $\P \Sym^2 W$ with 
  $\P\Sym^2 \C^{k+1}$ such that $q = I_s$.
  By \cref{lem:intrinsicDescription}, the coordinate-wise square 
  $L^{\circ 2} $ is a linear re-embedding of the variety obtained 
  by the projection of $\nu_2(\P W)$ from the point $q=I_s 
  \in \P \Sym^2 W$. Note that $V(J)$ describes the set of points $Y\in \P 
  \Sym^2 W$ lying on the line joining $q $ with some point in $\nu_2(\P 
  W)$. Hence, the projection from $q$ is 
  given by intersecting $V(J)$ with a hyperplane $H \subset \P \Sym^2 W$ not 
  containing $q = I_s$.
  
  From \cref{prop:matrixCompletion}, we know that $V(J) \cap H$ has a 
  set-theoretic description inside $H \cong \P^{\binom{k+2}{2}-2}$ as the zero set of the indicated number of quadric and cubic forms. The coordinate-wise square $L^{\circ 
  2}$ is by \cref{lem:intrinsicDescription} the image of $V(J) 
  \cap H$ under a linear embedding $\vartheta\colon H \hookrightarrow \P^n$, 
  leading to additional $n-\binom{k+2}{2}+2$ linear forms vanishing on 
  $L^{\circ 2}$.
\end{proof}

We prove \cref{prop:matrixCompletion} in several steps. First, we 
describe a set $\mathcal{X}$ of certain low-degree polynomials in the 
ideal $J$. Secondly, we show that $V(\mathcal{X}) = V(J)$. Finally, we identify 
a subset of $\mathcal{X}$ providing minimal generators of the ideal $(\mathcal 
X) \subset \Cy$, consisting of the claimed number of quadratic and cubic 
forms.

\begin{lemma}
  The following sets of polynomials in $\Cy$ are contained in the ideal $J$:
  \begin{align*}
  \mathcal{E} &:= \{Y_{ij|\ell m} \:\mid \:\{i,j\} \cap \{\ell,m\} \subset 
  \{s+1, \ldots,k+1\}\}, \\
  \mathcal{F} &:= \{Y_{i\ell|i m} - Y_{j\ell|jm} \:\mid\: i, j \leq s,\: 
  \{\ell\} 
  \cap \{m\} \subset \{s+1, \ldots,k+1\}\}, \\
  \mathcal{G} &:= 
  \{Y_{ij|ij} - Y_{j\ell|j\ell} + Y_{\ell m|\ell m} - Y_{m i|m i} \:\mid\: 
  i, j, \ell, m \leq s \text{ distinct}\},\\
  \mathcal{H}_1 &:= 
  \{y_{i\ell}(Y_{ij|ij}-Y_{i\ell|i\ell})-(y_{\ell \ell}-y_{jj})Y_{ij|j\ell} 
  \mid 
  i,j,\ell \leq s\}, \\
  \mathcal{H}_2 &:=    
  \{(y_{ii}-y_{jj})Y_{ij|ij}+(y_{jj}-y_{\ell \ell})Y_{j \ell|j \ell}+(y_{\ell 
    \ell}-y_{ii})Y_{\ell i|\ell i}
  \mid i,j, \ell \leq s\}.
\end{align*}  
\end{lemma}

\begin{proof}
  Using that
  \begin{equation}\label{eq:minorsOfMY}
  \begin{aligned}
      Y_{ij|ij} &= M_{ij|ij}-(y_{ii}+y_{jj})t-t^2 \qquad \text{for all $i,j 
      \leq s$ 
      distinct and} \\
      Y_{i\ell|j\ell} &= M_{i \ell|j \ell} - ty_{ij} \qquad \hspace{2.11cm} 
      \text{for all $\ell 
      \leq s$, $\{i\} \cap \{j\} \subset \{s+1,\ldots,k+1\}$},
  \end{aligned}
  \end{equation}
  we can check that
  \begingroup \footnotesize
  \begin{align*}
    Y_{ij|\ell m} &= M_{ij|\ell m}, \\
    Y_{i\ell|i m} - Y_{j\ell|jm} &= M_{i\ell|i m} - M_{j\ell|jm}, \\
    Y_{ij|ij} - Y_{j\ell|j\ell} + Y_{\ell m|\ell m} - Y_{m i|m i} &= M_{ij|ij} 
    - M_{j\ell|j\ell} + M_{\ell m|\ell m} - M_{m i|m 
    i}, \\
    y_{i\ell}(Y_{ij|ij}-Y_{i\ell|i\ell})-(y_{\ell \ell}-y_{jj})Y_{ij|j\ell} &= 
     y_{i\ell}(M_{ij|ij}-M_{i\ell|i\ell})-(y_{\ell \ell}-y_{jj})M_{ij|j\ell}, \\
    (y_{ii}-y_{jj})Y_{ij|ij}+(y_{jj}-y_{\ell \ell})Y_{j \ell|j \ell}+(y_{\ell 
      \ell}-y_{ii})Y_{\ell i|\ell i} &= 
      (y_{ii}-y_{jj})M_{ij|ij}+(y_{jj}-y_{\ell \ell})M_{j \ell|j 
      \ell}+(y_{\ell \ell}-y_{ii})M_{\ell i|\ell i}
  \end{align*}
  \endgroup
  holds for respective indices $i,j,\ell,m$. From this, we conclude 
  that these polynomials are contained in $J_0 \cap \Cy = J$.
\end{proof}

From now on, we denote
$\mathcal{X} := \mathcal{E} \cup \mathcal{F} \cup \mathcal{G} \cup 
\mathcal{H}_1 \cup \mathcal{H}_2$. These polynomials describe $V(J)$:
%As a next step, we show that the polynomials in $\mathcal X$ cut out $V(J)$.

\begin{lemma} \label{lem:setTheoreticEquality}
  Inside $\P \Sym^2 \C^{k+1} = \P^{\binom{k+2}{2}-1}$, we consider the open sets
    \[U_1 := \P \Sym^2 \C^{k+1}  \setminus \{I_s\} \qquad \text{and} \qquad U_2 
    := 
    \P \Sym^2 \C^{k+1} \setminus
    \left\{\left(\begin{smallmatrix} \begin{smallmatrix}* & * \\ * & * 
    \end{smallmatrix} & 0 \\ 0 & 0\end{smallmatrix}\right)\right\}.\]
  \begin{enumerate}[(i)]
    \item If $s \geq 3$, then $V(\mathcal X)$ and $V(J)$ agree 
    scheme-theoretically on $U_1$.
    \item If $s = 2$, then $V(\mathcal X)$ and $V(J)$ agree 
    scheme-theoretically on $U_2$.
    \item For $s$ arbitrary, $V(\mathcal X)$ and $V(J)$ agree set-theoretically.
  \end{enumerate}
\end{lemma}

\begin{proof}
  For $k \leq 5$, we have checked computationally with 
  a straightforward implementation in 
  \texttt{Macaulay2} \cite{M2} that even the ideal-theoretic 
  equality 
  $(\mathcal X) = J$ holds.
  We now argue that 
  from this we can conclude the claim for arbitrary $k$.
  \begin{enumerate}[(i)]
    \item Let $s \geq 3$. We need to show that the ideal generated by $\mathcal 
    X 
    \subset \Cy$ 
    coincides with $J \subset 
    \Cy$ after localisation at any element in the set 
    \[\{y_{ij} \mid \{i\} \cap 
    \{j\} \subset \{s+1, \ldots,k+1\}\} \cup \{y_{ii}-y_{jj} \mid i,j \leq 
    s\},\]
    since
    the union of the corresponding non-vanishing sets $D(y_{ij}), 
    D(y_{ii}-y_{jj})$
    % \subset \P \Sym^2 \C^{k+1}$
    is $U_1$.
    
    In order to show that $(\mathcal X)$ and $J$ agree after localisation at 
    $y_{i_0j_0}$ for $\{i_0\} \cap \{j_0\} \subset \{s+1,\ldots,k+1\}$, we may 
    substitute $y_{i_0 j_0} = 1$ in both ideals. For a fixed $\ell_0 
    \leq s$ distinct from $i_0$ and $j_0$, we note that
    $t+Y_{i_0\ell_0|j_0\ell_0} = M_{i_0 \ell_0| j_0 \ell_0} \in 
    \restr{J_0}{y_{i_0 j_0}=1}.$
    Hence, eliminating $t$ from $\restr{J_0}{y_{i_0 j_0}=1}$ just amounts to 
    replacing $t = -Y_{i_0 \ell_0|j_0 \ell_0}$ in each occurrence of $t$ in the 
    minors $M_{ij|\ell m}$ (for $i \neq j$, $\ell \neq m$) generating the ideal 
    $J_0$.
    
    According to \eqref{eq:minorsOfMY}, this leads to the following generators 
    of 
    $\restr{J}{y_{i_0 j_0} = 1}$:
    \begin{itemize}
     \item 
     \begin{tabular}{>{\raggedright}p{0.4\textwidth}>{\raggedleft}p{0.4\textwidth}}
      $Y_{i_0 
     \ell_0|j_0\ell_0}^2-(y_{ii}+y_{jj})Y_{i_0 
      \ell_0|j_0\ell_0}+Y_{ij|ij}$ & for $i\neq j \leq s$, 
     \end{tabular} 
      \vspace{0.1em}
     \item 
     \begin{tabular}{>{\raggedright}p{0.3\textwidth}>{\raggedleft}p{0.5\textwidth}}
       $-y_{ij}Y_{i_0 \ell_0|j_0\ell_0}+Y_{i \ell|j\ell}$ & for 
      $\ell \leq s, \{i\} \cap \{j\} \subset \{s+1,\ldots,k+1\}$, %\\[0.3em]
     \end{tabular}
      \vspace{0.1em}
     \item 
     \begin{tabular}{>{\raggedright}p{0.1\textwidth}>{\raggedleft}p{0.7\textwidth}}
     $Y_{ij|\ell m}$ & for %$i\neq j$, $\ell \neq m$ s.t.\ 
      $\{i,j\}\cap \{\ell,m\} 
      \subset\{s+1,\ldots,k+1\}$.
     \end{tabular}
    \end{itemize}
%    \begin{equation} \label{eq:gensOfLocalization}
%    \begin{aligned}
%    &Y_{i_0 \ell_0|j_0\ell_0}^2-(y_{ii}+y_{jj})Y_{i_0 
%      \ell_0|j_0\ell_0}+Y_{ij|ij} \qquad \text{for $i\neq j \leq s$}, \\[0.3em]
%    &-y_{ij}Y_{i_0 \ell_0|j_0\ell_0}+Y_{i \ell|j\ell} \qquad \text{for 
%      $\ell \leq s, \{i\} \cap \{j\} \subset \{s+1,\ldots,k+1\}$}, \\[0.3em]
%    &Y_{ij|\ell m} \qquad \text{for $i\neq j$, $\ell \neq m$ s.t.\ 
%      $\{i,j\}\cap \{\ell,m\}, 
%      \subset\{s+1,\ldots,k+1\}$}.
%    \end{aligned}
%    \end{equation}
    To check that $\restr{J}{y_{i_0 j_0} = 1} = \restr{(\mathcal X)}{y_{i_0 
    j_0} = 1}$, we need to check that each of 
    %the polynomials in \eqref{eq:gensOfLocalization} 
    these polynomials
    belong to $\restr{(\mathcal X)}{y_{i_0 j_0} = 
    1}$. For this, it is 
    enough 
    to see that they can be expressed in terms of those polynomials in 
    $\mathcal 
    X$ 
    that only involve variables with indices among 
    $\{i_0,j_0,\ell_0,i,j,\ell\}$. 
    This corresponds to showing the claim for a corresponding symmetric 
    submatrix 
    of $M$ of size at most $6 \times 6$. We conclude that it is enough to check 
    $\restr{J}{y_{i_0 j_0} = 1} = \restr{(\mathcal X)}{y_{i_0 j_0} = 1}$ for $k 
    \leq 5$.
    
    Similarly, in order to show that
    $\restr{J}{y_{i_0 i_0}-y_{j_0 j_0}=1} = \restr{(\mathcal X)}{y_{i_0 
        i_0}-y_{j_0 j_0}=1}$ holds for $i_0,j_0 \leq s$ distinct, we realise 
        that
    $t+Y_{i_0 \ell_0|i_0 \ell_0}-Y_{j_0 \ell_0|j_0 \ell_0} = M_{i_0 \ell_0|i_0 
    \ell_0} - M_{j_0 \ell_0|j_0 \ell_0} \in \restr{J_0}{y_{i_0 j_0}=1}$
    holds for fixed $\ell_0 \leq s$ distinct from $i_0$ and $j_0$. 
    Therefore, replacing $t = Y_{j_0 \ell_0|j_0 \ell_0}-Y_{i_0 
    \ell_0|i_0 \ell_0}$ in the expressions for the $2 \times 2$-minors of $M$ 
    describes generators of $\restr{J}{y_{i_0 i_0}-y_{j_0 j_0}=1}$. As 
    before, these polynomials involve variables with at most six distinct 
    indices, so it is enough to verify the claim for $k \leq 5$ by the 
    same argument as above.
    
    \item For $s = 2$, the argument from~(i) still shows $\restr{J_0}{y_{i_0 
        j_0} = 1} = \restr{(\mathcal X)}{y_{i_0 j_0} = 1}$ for $\{i_0, j_0\} 
        \cap \{3,\ldots,k+1\} \neq \emptyset$. For the 
    localisation at $y_{12}$ and at $y_{11}-y_{22}$, the 
    argument does not apply since we cannot choose $\ell_0$ distinct from 
    $\{i_0,j_0\} = \{1,2\}$ as before. Hence, we have shown the equality of 
    $V(\mathcal X)$ and $V(J)$ only on $U_2$.
    
    \item We observe that the polynomials in $\mathcal X$ vanish on the point 
    $I_s \in \P\Sym^2 \C^{k+1}$, and that $I_s \in V(J)$ by definition of 
    %the ideal 
    $J$. Together with (i), this proves the claim for $s \geq 3$.
    
    For $s = 2$, 
    %it is straightforward to check that
    the polynomials in 
    $\mathcal X$ vanish on all symmetric matrices of the form
    $A = \left(\begin{smallmatrix} \begin{smallmatrix}a & c \\ c & b 
    \end{smallmatrix} & 0 \\ 0 & 0\end{smallmatrix}\right) \in \Sym^2 
    \C^{k+1}$. On the other 
    hand, each such matrix is a point in $V(J)$, since $A+t_0 I_2$ is a 
    matrix of rank 
    $\leq 1$ for $t_0 \in \C$ such that $t_0^2+(a+b)t_0+(ab-c^2)=0$. Together 
    with (ii), we conclude that $V(\mathcal X) = V(J)$ holds 
    set-theoretically. \qedhere
  \end{enumerate}
\end{proof}

\begin{lemma} \label{lem:linearIndependenceAmongSets}
  The vector spaces spanned by the polynomials in $\mathcal X$ satisfy:
  \begin{enumerate}[(i)]
    \item $\langle \mathcal E \cup \mathcal F \cup \mathcal G 
      \rangle = \langle \mathcal E \rangle \oplus \langle \mathcal F \rangle 
      \oplus 
      \langle \mathcal G \rangle,$
    \item $\langle \mathcal H_1 \cup 
    \mathcal H_2 \rangle \cap (\mathcal E, \mathcal F, \mathcal G) = \emptyset$ 
    for $s = 3$,
    \item $\mathcal H_1 \cup \mathcal H_2 \subset (\mathcal E, \mathcal F, 
    \mathcal G)$ for $s \neq 3$.
  \end{enumerate}
\end{lemma}

\begin{proof}
  Let $\mathcal M_{\mathcal E} \subset \Cy$ denote the set of monomials 
  occurring in one of the polynomials of $\mathcal E$, and analogously for 
  $\mathcal F$, $\mathcal G$, $\mathcal H_1$ and $\mathcal H_2$.
  \begin{enumerate}[(i)]
  \item This follows from the observation that $\mathcal{M}_{\mathcal E}$, $\mathcal{M}_{\mathcal 
  F}$ and $\mathcal{M}_{\mathcal G}$ are disjoint sets.
  % It follows that $\langle \mathcal E 
  %\cup \mathcal F \cup \mathcal G \rangle = \langle \mathcal E \rangle \oplus 
  %\langle \mathcal F \rangle \oplus \langle \mathcal G \rangle$.
  \item For $s = 3$, note that $\mathcal G = \emptyset$ and none of the 
  monomials in $\mathcal M_{\mathcal E} \cup \mathcal M_{\mathcal F}$ is of 
  the form $y_{ij}y_{\ell m}$ with $i,j,\ell,m \leq 3$. On the other hand, the  
  monomials in $\mathcal M_{\mathcal H_1} \cup \mathcal M_{\mathcal H_2}$ 
  are of the form $y_{i_1 j_1} y_{i_2 j_2} y_{i_3 j_3}$ with $i_1, i_2, 
  i_3, j_1, j_2, j_3 \leq 3$. Hence, no monomial in $\mathcal M_{\mathcal H_1} 
  \cup \mathcal M_{\mathcal H_2}$ is a multiple of any of the monomials in 
  $\mathcal M_{\mathcal E} \cup \mathcal M_{\mathcal F}$, so $\langle \mathcal 
  H_1 \cup 
  \mathcal H_2 \rangle \cap (\mathcal E, \mathcal F, \mathcal G) = \emptyset$.
  \item If $s = 2$ we have $\mathcal H_1 \cup \mathcal H_2=\emptyset$, so the 
  claim 
  is trivial. Let $s \geq 4$. Then for all $i,j,\ell,m \leq s$ distinct, we have
  \begingroup \small
  \begin{align*}
  &\hspace{0.75cm}y_{i\ell}(Y_{ij|ij}-Y_{i\ell|i\ell})-(y_{\ell 
  \ell}-y_{jj})Y_{ij|j\ell} \\
  &\hspace{0.75cm}= -2y_{j m}Y_{i j| \ell m}-y_{j m}Y_{i m| j \ell} -y_{i 
  \ell}(Y_{i 
    \ell| 
    i 
    \ell}-Y_{\ell j| \ell j}+Y_{j m| j 
    m}-Y_{m i| m i})+y_{i m}(Y_{j \ell| j m}-Y_{i \ell| i m})\\
  &\hspace{0.75cm}\phantom{{}=} +y_{i j}(Y_{i j| i 
    \ell}-Y_{m j| 
    m \ell})+(y_{i i}-y_{j j})(Y_{i j| \ell j}-Y_{i m| \ell m})-y_{j 
    \ell}(Y_{i 
    \ell| j \ell}-Y_{i m| j m}) \in (\mathcal E, \mathcal F, \mathcal G), 
    \\[0.3em]
  &\hspace{0.75cm}(y_{ii}-y_{jj})Y_{ij|ij}+(y_{jj}-y_{\ell \ell})Y_{j \ell|j 
  \ell}+(y_{\ell 
    \ell}-y_{ii})Y_{\ell i|\ell i} \\
  &\hspace{0.75cm}= (y_{i i}-y_{j j})(Y_{i j| i 
    j}-Y_{j \ell| j \ell}+Y_{\ell m| \ell m}-Y_{m i| m i})+(y_{\ell 
    \ell}-y_{i i})(Y_{i \ell| i \ell}-Y_{\ell j| \ell j}+Y_{j m| j m}-Y_{m i| 
    m i})\\
  &\hspace{0.75cm}\phantom{{}=}+y_{\ell m}(Y_{i \ell| i m}-Y_{j \ell| j 
  m})-y_{j m}(Y_{i j| 
    i m}-Y_{\ell j| \ell m})+y_{i m}(Y_{j i| j m}-Y_{\ell i| \ell m}) \in 
  (\mathcal E, \mathcal F, \mathcal G),
  \end{align*}
  \endgroup
  so $\mathcal H_1$ and $\mathcal H_2$ lie in the ideal generated by $\mathcal 
  E$, $\mathcal F$ and $\mathcal G$.
  \qedhere
  \end{enumerate}
\end{proof}

Next, we identify maximal linearly independent subsets of $\mathcal E$, 
$\mathcal F$, $\mathcal G$.

\begin{lemma} \label{lem:basesEFG}
  The following sets are bases for the vector spaces $\langle \mathcal E 
  \rangle$, $\langle \mathcal F \rangle$ and $\langle \mathcal G \rangle$:
  \begingroup \small
  \begin{align*}
    \mathcal B_{\mathcal E} &:= \{Y_{ij|\ell m} \mid i < j, \ell < m, i 
    \leq 
    \ell \leq j \text{ s.t.\ } \{i,j\} \cap \{\ell,m\} \subset \{s+1, \ldots, 
    k+1\} \text{ and } j \leq m \text{ if } i=\ell\}, \\
    \mathcal B_{\mathcal F} &:= \{Y_{i\ell|i m} - Y_{1\ell|1m} \mid 2 
    \leq i 
    \leq s, \: 2 \leq \ell \leq m \text{ s.t. } i \notin 
    \{\ell, m\}, \: \{\ell\} \cap \{m\} \subset \{s+1, \ldots,k+1\}\} \\
         &\phantom{{}:={}}\cup \{Y_{i1|im} - Y_{21|2m} \mid 3 \leq i \leq s,\:  
         3 \leq m \leq k+1,\: i \neq m\}
        \; \cup\; \{Y_{i1|i2} - Y_{31|32} \mid i \in 
        \{4,\ldots,s\}\}, \\
    \mathcal B_{\mathcal G} &:= \{Y_{12|12} - Y_{2\ell|2\ell} + Y_{\ell m|\ell 
    m} - Y_{m1|m1} \:\mid\: 3 \leq m \leq s-1,\: \ell \in \{3,4,\ldots,m-1\} 
    \cup 
    \{s\}\} \\
    &\phantom{{}:={}}\cup \{Y_{1s|1s} - Y_{s2|s2} + Y_{2m|2m} - Y_{m1|m1} 
    \:\mid\: 3 \leq m \leq s-1\}.
%    \mathcal B_{\mathcal G} &:= \{Y_{12|12} - Y_{2\ell|2\ell} + Y_{\ell m|\ell 
%    m} - Y_{m1|m1} 
%    \:\mid\: 3 \leq \ell < m \leq s-1\} \\
%        &\phantom{{}:={}}\cup \{Y_{12|12} - Y_{2s|2s} + Y_{sm|sm} - Y_{m1|m1} 
%        \:\mid\: 3 \leq m \leq s-1\} \\
%        &\phantom{{}:={}}\cup \{Y_{1s|1s} - Y_{s2|s2} + Y_{2m|2m} - Y_{m1|m1} 
%        \:\mid\: 3 \leq m \leq s-1\}.
  \end{align*}
  \endgroup
\end{lemma}

\begin{proof}
  The polynomials in $\mathcal E$ not contained in $\mathcal 
  B_{\mathcal E} \cup (- \mathcal B_{\mathcal E})$ are the polynomials 
  $Y_{ij|\ell m}$ 
  for 
  $i < j < 
  \ell < m$. However, these can be expressed as $Y_{ij|\ell m} = Y_{i \ell|j 
  m}-Y_{im|j \ell} \in \langle \mathcal B_{\mathcal E}\rangle$.
  Hence $\mathcal B_{\mathcal E}$ spans $\langle \mathcal E \rangle$. For $i < 
  j$, $\ell < m$ with $i \leq \ell \leq j$ such that $\{i,j\} \cap 
  \{\ell,m\} \subset \{s+1, \ldots, k+1\}$, we note that $Y_{ij|\ell m} \in 
  \Cy$ is the unique polynomial in $\mathcal B_{\mathcal 
  E}$ containing the monomial $y_{im} y_{\ell j}$. In particular, the 
  polynomials in $\mathcal B_{\mathcal E}$ are linearly independent, so 
  $\mathcal B_{\mathcal E}$ forms a basis of $\langle \mathcal E \rangle$.
  
  If $i,j,\ell,m \in \{1,\ldots,k+1\}$ with $\ell < m$ are such that 
  $Y_{i\ell|i m} - Y_{j\ell|jm} \in \mathcal F \setminus (\mathcal B_{\mathcal 
  F} \cup -\mathcal B_{\mathcal F})$, then
    \[Y_{i\ell|i m} - Y_{j\ell|jm} = \begin{cases}
    (Y_{i\ell|i m} - Y_{1\ell|1m})-(Y_{j\ell|j m} - Y_{1\ell|1m}) &\text{if }
    \ell, m \neq 1, \\
    (Y_{i1|im} - Y_{21|2m})-(Y_{j1|jm} - Y_{21|2m}) &\text{if }
    \ell = 1, m \neq 2, \\
    (Y_{i1|i2} - Y_{31|32})-(Y_{j1|j2} - Y_{31|32}) &\text{if }
    \ell = 1, m = 2,
    \end{cases}\]
  so $\mathcal B_{\mathcal F}$ spans $\langle \mathcal F \rangle$. Each of the  
  polynomials $Y_{i\ell|i m} - Y_{j\ell|jm}$ in $\mathcal B_{\mathcal F}$ 
  contains a monomial not occurring in any of the other polynomials of 
  $\mathcal B_{\mathcal F}$, namely $y_{ii} y_{\ell m}$.
  Therefore, the polynomials in $\mathcal B_{\mathcal F}$ are linearly 
  independent.
  
  For $3 \leq m \leq s-1$ and $\ell \in \{3,\ldots,m-1\} \cup \{s\}$, 
  the polynomial 
  $Y_{12|12} - Y_{2\ell|2 \ell} + Y_{\ell m|\ell m} - Y_{m1|m1}$
  is the unique polynomial in $\mathcal B_{\mathcal G}$ containing 
  the monomial $y_{\ell \ell} y_{m m}$. In particular, if a linear combination 
  of polynomials in $\mathcal B_{\mathcal G}$ is zero, none of the above 
  polynomials can occur in this linear combination. The remaining polynomials 
  in $\mathcal B_{\mathcal G}$ are of the form $Y_{1s|1s} - Y_{s2|s2} + 
  Y_{2m|2m} - Y_{m1|m1}$ for $3 \leq m \leq s-1$. Among these, the polynomial 
  containing the monomial $y_{22} y_{mm}$ is unique. We conclude that the 
  polynomials in $\mathcal B_{\mathcal G}$ are linearly independent.
  
  We observe that 
    \[\mathcal G \subset \big\{\sum_{i,j=1}^s a_{ij} Y_{ij|ij} \mid A=(a_{ij}) 
    \in 
    \C^{s \times s} \text{ symmetric with } a_{ii} = 0 \text{ and } 
    (1,\ldots,1) A 
    =0\big\}.\]
  The vector space of symmetric $s \times s$-matrices with zero diagonal and 
  whose columns all sum to zero is of dimension $\binom{s}{2}-s$, so $\dim_\C 
  \langle \mathcal G\rangle \leq \binom{s}{2}-s$. On the other hand, we can 
  count that $|\mathcal B_{\mathcal G}| = \binom{s-3}{2}+2(s-3) = 
  \binom{s}{2}-s$, so 
  %from the linear independence of the elements in $\mathcal 
  %B_{\mathcal G}$ we conclude that 
  $\mathcal B_{\mathcal G}$ is a basis of $\langle \mathcal G \rangle$.
\end{proof}

\begin{proof}[Proof of \cref{prop:matrixCompletion}]
  By \cref{lem:setTheoreticEquality}, $V(J) = V(\mathcal X)$ holds 
  set-theoretically. 
  For $s = 3$, we observe that $\mathcal H_1 \cup \mathcal H_2$ consists up to 
  sign of seven linearly independent cubics, so by 
  \cref{lem:linearIndependenceAmongSets}, the ideal $(\mathcal X)$ is in this 
  case minimally generated by those seven cubics and the polynomials in 
  $\mathcal B_{\mathcal E}$, $\mathcal B_{\mathcal F}$ and $\mathcal 
  B_{\mathcal G}$ from \cref{lem:basesEFG}.
  
  For $s \neq 3$, \cref{lem:linearIndependenceAmongSets} and 
  \cref{lem:basesEFG} show that $(\mathcal X)$ is minimally generated just by 
  the polynomials $\mathcal B_{\mathcal E} \cup \mathcal B_{\mathcal F} \cup 
  \mathcal B_{\mathcal G}$. Straightforward counting gives:
  \begin{align*}
    |\mathcal B_{\mathcal E}| &= {\textstyle 2 \binom{k+1}{4} + (k-s+1) 
    \binom{k}{2} + \binom{k-s+1}{2}} \\
    &=(k^4-6sk^2+4k^3+6s^2-6sk+5k^2-6s+2k)/12,
    \\
    |\mathcal B_{\mathcal F}| &= {\textstyle (s-1) 
    \left(\binom{k-1}{2}+(k-s+1)\right) + 
    (s-2)(k-2) + \binom{s-3}{1}} \\
    &= 
    \begin{cases}
      (sk^2-k^2+sk-2s^2-3k+4s-2)/2 &\text{if } s \geq 3, \\
      (sk^2-k^2+sk-2s^2-3k+4s)/2 &\text{if } s = 2,
    \end{cases} \\
    |\mathcal B_{\mathcal G}| &=
    \begin{cases}
      {\textstyle \binom{s}{2}-s = (s^2-3s)/2} &\text{if }s \geq 3, \\
      0 &\text{if } s=2.
    \end{cases}
  \end{align*}
  Adding up these cardinalities gives the claimed number of quadratic forms.
  %, concluding the proof.
%  \[|\mathcal B_{\mathcal E} \cup \mathcal B_{\mathcal F} \cup 
%  \mathcal B_{\mathcal G}| = 
%  \begin{cases}
%    (k+3)(k+2)(k+1)(k-2)/12 &\text{if } s \geq 3, \\
%    (k+3)(k+2)(k+1)(k-2)/12 + 2 &\text{if } s = 2.
%  \end{cases}\]
%  This concludes the proof.
\end{proof}

\begin{remark} \label{rem:schemeTheoretic}
  In fact, for $s \geq 3$, our proof 
  %shows not only that $V(J) = V(\mathcal X)$ 
  %holds 
  %set-theoretically, but in fact it 
  shows that $V(\mathcal X)$ is the same 
  scheme as $V(J)$ away from the point $I_s \in \P \Sym^2 \C^{k+1}$. In the 
  proof of \cref{thm:uniqueQuadricCase}, we considered $V(J) \cap H$, where $H$ 
  is a hyperplane not containing $I_s$. 
  Since $V(J) \cap H = V(\mathcal X) \cap H$ scheme-theoretically, we conclude 
  that our  
  equations for $L^{\circ 2}$ in \cref{thm:uniqueQuadricCase} describe 
  not only the correct set, but even the correct scheme. In fact, we believe 
  that we have ideal-theoretic equality for the specified set of polynomials, 
  but our proof stops short of verifying this.
\end{remark}

We now prove the
result about eigenspaces of symmetric matrices stated as \cref{cor:eigenvalues}. It follows directly from the proof of \cref{prop:matrixCompletion}.

\begin{proof}[Proof of \cref{cor:eigenvalues}]
  A complex symmetric matrix $A \in \C^{s \times s}$ has an eigenspace 
  of codimension~1 with respect to an eigenvalue $\lambda \in \C$ if and only 
  if the matrix $A-\lambda \id$ is of rank~1, which means that $A \in 
  V(J)$ for the case $s = k+1$. By \cref{lem:setTheoreticEquality} and \cref{lem:linearIndependenceAmongSets}, this is equivalent to the vanishing of the equations ${\mathcal E\cup \mathcal F \cup \mathcal G}$, which are the above relations among $2\times 2$-minors for $s = k+1 \geq 4$. The second claim was proved in \cref{prop:eigenvaluesEarly}.
\end{proof}

The proof of \cref{thm:uniqueQuadricCase} was based on relating the 
coordinate-wise square $L^{\circ 2}$ in the case $\dim_\C I(Z)_2 = 1$ to the 
question when a symmetric matrix can by completed to a rank~1 matrix by adding 
a multiple of $I_s$. In the same spirit, for \emph{arbitrary} linear spaces $L $ 
(no restrictions on the set of quadrics containing $Z $), 
determining the ideal of the coordinate-wise square $L^{\circ 2}$ 
boils down to the following problem in symmetric rank~1 matrix completion:

\begin{problem} \label{prob:MatrixProblem}
  For a fixed matrix $B \in \C^{(n+1) \times (k+1)}$ of rank $k+1$, find the 
  defining equations of the set
  \[\left\{M \in \C^{(k+1)\times(k+1)} \text{ symmetric } \mid 
  \begingroup \footnotesize \begin{array}{c}
  \exists P \in \C^{(k+1)\times(k+1)} \text{ symmetric such that } 
    BPB^T \\ \text{ has a zero diagonal and } \rank(M+P) = 1
  \end{array} \endgroup \right\}.\]
\end{problem}

Indeed, let $L $ be an arbitrary linear space of dimension~$k$ and 
let $B \in \C^{(n+1) \times (k+1)}$ be a chosen matrix of full rank describing 
$L$ as the image of the linear embedding $\P^k \hookrightarrow \P^n$ given by 
$B$. Then the rows of $B$ form the finite set of points $Z \subset (\P^k)^*$. 
Identifying quadratic forms on $\P^k$ with symmetric $(k+1)\times 
(k+1)$-matrices, the subspace $I(Z)_2 \subset \Sym^2 (\C^{k+1})^*$ 
corresponds to
  \[I(Z)_2 = \{P \in \C^{(k+1)\times(k+1)} \text{ symmetric such that } 
  BPB^T \\ \text{ has a zero diagonal}\}.\]
By \cref{lem:intrinsicDescription}, the coordinate-wise square $L^{\circ 2} 
$ is a linear re-embedding of projecting the second Veronese variety
  \[\nu_2(\P^k) = \{\text{rank~$1$ symmetric $(k+1)\times (k+1)$-matrices up to 
  scaling}\}\]
from $\P(I(Z)_2)$, so describing the ideal of $L^{\circ 2}$ 
%essentially 
corresponds to solving \cref{prob:MatrixProblem} for the given matrix $B$.
Similarly, describing the coordinate-wise $r$-th power of a linear space 
corresponds to the
analogous
problem in symmetric rank~$1$ tensor completion.
%analogous to \cref{prob:MatrixProblem}.

\medskip

By \cref{lem:intrinsicDescription}, determining the coordinate-wise 
$r$-th power of a linear space corresponds to describing the projection of the 
$r$-th Veronese variety from a linear space of the form $\P(I(Z)_r)$ for a 
non-degenerate finite set of points $Z$. We may 
ask how general this problem is, and pose the question which linear subspaces 
of $\P \Sym^r W$ are of the form $\P(I(Z)_r)$:

%\begin{question}
%  Let $W$ be a $k$-dimensional vector space. Which linear subspaces of $\P 
%  \Sym^r W$ are of the form $\P(I(Z)_r)$ for a non-degenerate finite set of 
%  points $Z \subset \P W^*$ of cardinality $\leq n+1$?
%\end{question}

\begin{question} \label{qu:whichProjections}
  Which linear subspaces of $\C[z_0,\ldots,z_k]_r$ can be realised as the set of
  degree~$r$ polynomials vanishing on some non-degenerate finite set of points 
  in $\P^k$ of cardinality $\leq n+1$?
\end{question}

We envision that 
%via \cref{lem:intrinsicDescription}, an answer to \cref{qu:whichProjections} 
an answer to this question may lead to insights into describing 
which varieties can occur as the coordinate-wise $r$-th power of some linear 
space in $\P^n$.

%\printbibliography

\begin{small}

\end{small}

\end{document}